\documentclass[11pt]{amsart}
\usepackage{mathtools, amssymb,amsfonts,amsthm,amscd,amsmath,amsgen,upref,enumerate,nicefrac,esint}

\usepackage[margin=1in]{geometry}

\theoremstyle{plain}
\newtheorem{cor}{Corollary}

\newtheorem{prop}[cor]{Proposition}

\newtheorem{thm}[cor]{Theorem}
\newtheorem*{thm*}{Theorem}
\theoremstyle{definition}

\newtheorem{remark}[cor]{Remark}

\numberwithin{cor}{section}
\numberwithin{equation}{section}

\DeclareMathOperator{\C}{C}

\DeclareMathOperator{\Supp}{Supp}

\newcommand{\E}{\mathbb{E}}

\renewcommand{\d}{\delta}

\renewcommand{\and}{\quad\textrm{ and }\quad}

\renewcommand{\P}{\mathbb{P}}
\renewcommand{\a}{\alpha}

\renewcommand{\o}{\omega}
\renewcommand{\O}{\Omega}

\newcommand{\F}{\mathcal F}

\newcommand{\Exc}{\textrm{Exc}}

\newcommand{\R}{\mathbb{R}}

\newcommand{\N}{\mathbb{N}}

\newcommand{\norm}[1]{\left\| #1 \right\|}

\newcommand{\ve}{\varepsilon}

\newcommand{\abs}[1]{\left|#1\right|}
\providecommand{\ud}[1]{\, \mathrm{d} #1}
\providecommand{\dx}{\ud{x}}
\providecommand{\dy}{\ud{y}}

\providecommand{\dr}{\ud{r}}

\providecommand{\ds}{\ud{s}}
\providecommand{\dt}{\ud{t}}

\providecommand{\dd}{\ud}

\def\XXint#1#2#3{{\setbox0=\hbox{$#1{#2#3}{\int}$ }
\vcenter{\hbox{$#2#3$ }}\kern-.6\wd0}}

\usepackage{graphicx}
\usepackage{color}

\begin{document}

\title{Large-scale regularity in stochastic homogenization with divergence-free drift}

\author{Benjamin Fehrman}

\begin{abstract}
We provide a simple proof of quenched stochastic homogenization for random environments with a mean zero, divergence-free drift under the assumption that the drift admits a stationary $L^d$-integrable stream matrix in $d\geq 3$ or an $L^{2+\d}$-integrable stream matrix in $d=2$.  In addition, we prove that the environment almost surely satisfies a large-scale H\"older regularity estimate and first-order Liouville principle.
\end{abstract}

\maketitle

\section{Introduction}

In this paper, we prove the quenched homogenization of the equation
\begin{equation}\label{intro_eq}-\nabla\cdot a(\nicefrac{x}{\ve},\o)\nabla u^\ve+\frac{1}{\ve}b(\nicefrac{x}{\ve},\o)\cdot \nabla u^\ve=f\;\;\textrm{in}\;\;U\;\;\textrm{with}\;\;u^\ve=g\;\;\textrm{on}\;\;\partial U,\end{equation}
for a uniformly elliptic matrix $a$ and a mean zero, divergence-free drift $b$.  The coefficients are jointly-measurable, stationary, and ergodic random variables defined on some probability space $(\O,\F,\P)$.  Stationarity asserts that the random environment is statistically homogenous:  there exists a measure preserving transformation group $\{\tau_x\colon\O\rightarrow\O\}_{x\in\R^d}$ such that
\begin{equation}\label{intro_rvs}(a(x,\omega),b(x,\omega))=(A(\tau_x\omega),B(\tau_x\omega))\;\textrm{for random variables}\;A\colon\O\rightarrow\R^{d\times d}, B\colon\O\rightarrow\R^d.\end{equation}
The ergodicity is a qualitative form of mixing:  for $g\in L^\infty(\O)$,
\begin{equation}\label{intro_ergodic}
g(\omega)=g(\tau_x\omega)\;\;\textrm{almost surely for every}\;\;x\in\R^d\;\;\textrm{if and only if}\;\;g\;\;\textrm{is almost surely constant.}
\end{equation}
In terms of the coefficients, we will assume that the matrix $A$ is bounded and uniformly elliptic:  there exist $\lambda,\Lambda\in(0,\infty)$ such that, almost surely for every $\xi\in\R^d$,
\begin{equation}\label{intro_ue}\abs{A\xi}\leq \Lambda\abs{\xi}\;\;\textrm{and}\;\;A\xi\cdot \xi\geq \lambda\abs{\xi}^2.\end{equation}
We will assume that, for some $\d\in(0,1)$, the random drift $B$ admits a stationary $L^{d\vee(2+\d)}$-integrable stream matrix:  there exists a skew-symmetric $S\in L^{d\vee(2+\d)}(\O;\R^{d\times d})$ which satisfies
\begin{equation}\label{intro_stream} \nabla\cdot S=B\;\;\textrm{for}\;\;(\nabla\cdot S)_i=\partial_kS_{ik},\end{equation}
fixed by the choice of gauge
\begin{equation}\label{intro_gauge}-\Delta S_{ij}=\partial_iB_j-\partial_j B_i.\end{equation}
We prove in Proposition~\ref{prop_stream} below that a stream matrix exists in $d\geq 3$ if $B$ is $L^d$-integrable and satisfies a finite range of dependence.  A stationary stream matrix does not exist in general if $d=2$, and for this reason the homogenization of \eqref{intro_eq} in $d=2$ remains largely an open problem.

In the symmetric case, for sufficiently regular coefficients and a sufficiently regular domain, solutions of \eqref{intro_eq} are related by the Feynman-Kac formula to a rescaling of the stochastic differential equation
\begin{equation}\label{intro_sde}\dd X_t=\sigma\left(X_t,\omega\right)\dd B_t+\left(\nabla\cdot a^t\left(X_t,\omega\right)-b\left(X_t,\omega\right)\right)\dt,\end{equation}
for $\sigma\sigma^t=2a$.  Precisely, for the exit time $\tau^\ve$ from the dilated domain $\nicefrac{U}{\ve}$,
\[u^\ve(x)=E_{\frac{x}{\ve},\omega}\left[g\left(\ve X_{\tau^\ve}\right)+\ve^2\int_0^{\tau^\ve}f\left(\ve X_s\right)\ds\right].\]
The homogenization of \eqref{intro_eq} is therefore equivalent to characterizing the asymptotic behavior of the exit distributions and exit times of \eqref{intro_sde} from large domains.  Furthermore, in the case $A=I$, equation \eqref{intro_sde} is the passive tracer model
\[\dd X_t=\sqrt{2}\dd B_t+b\left(X_t,\omega\right)\dt,\]
which is a simple approximation for the transport of a passive tracer particle in a turbulent, incompressible flow.  This model has applications to hydrology, meteorological sciences, and oceanography, and we point the reader, for instance, to Csanady \cite{Csa1973}, Frish \cite{Fri1995}, and Monin and Yaglom \cite{MonYag2007,MonYag2007II} for more details.

The stream matrix allows equation \eqref{intro_eq} to be rewritten in the form
\begin{equation}\label{intro_transform}-\nabla\cdot \left(a^\ve+s^\ve\right)\cdot\nabla u^\ve=f\;\;\textrm{in}\;\;U\;\;\textrm{with}\;\;u^\ve=g\;\;\textrm{on}\;\;\partial U,\end{equation}
for $s^\ve(x,\omega)=S(\tau_{\nicefrac{x}{\ve}}\omega)$.  The transformation \eqref{intro_transform} formally justifies the two-scale expansion familiar from the homogenization of divergence-form equations without drift.  That is, for the standard orthonormal basis $\{e_i\}_{i\in\{1,\ldots,d\}}$ of $\R^d$, we expect the corrector $\phi_i$ to be defined by a stationary gradient $\nabla\phi_i$ that solves
\begin{equation}\label{intro_corrector}-\nabla\cdot \left(A+S\right)\left(\nabla\phi_i+e_i\right)=0,\end{equation}
and we expect the homogenized coefficient $\overline{a}\in \R^{d\times d}$ to be defined by
\[\overline{a}e_i=\E\left[(A+S)(\nabla\phi_i+e_i)\right].\]
This is the case if $S$ is bounded, for which the methods of Papanicolaou and Varadhan \cite{PapVar1981} and Osada \cite{Osa1983} yield readily that, for the solution
\begin{equation}\label{intro_homogenized}-\nabla\cdot \overline{a}\nabla v = f \;\;\textrm{in}\;\;U\;\;\textrm{with}\;\;v=g\;\;\textrm{on}\;\;\partial U,\end{equation}
we have weak $H^1$-convergence of $u^\ve$ to $v$, and strong $H^1$-convergence of the two-scale expansion in the sense that almost surely
\[\lim_{\ve\rightarrow 0}\norm{u^\ve-(v+\ve\phi^\ve_i\partial_i v)}_{H^1(U)}=0,\]
for the rescaled correctors $\phi^\ve_i=\phi_i(\nicefrac{x}{\ve},\omega)$, where here and throughout the paper we use Einstein's summation convention over repeated indices.

The case of an unbounded stream matrix $S$ is fundamentally different.  Proving the existence of a solution to \eqref{intro_corrector} is essentially straightforward, arguing by approximation and the skew-symmetry of $S$.  Uniqueness is however not clear and was posed as an open problem in Avellaneda and Majda \cite{AveMaj1991}.  The reason for this is that, while the equation defines $S\nabla\phi_i$ as an element of the dual for any solution $\nabla\phi_i$, it is not clear that this rule extends to a skew-symmetric operator on the solution space.  Issues related to this fact explain the strong regularity assumptions used in Oelschl\"ager \cite{Oel1988} and form the technical core of the more recent works Kozma and T\'oth \cite{KozTot2017} and T\'oth \cite{Tot2018}.  In this paper we take a different approach based on the methods of \cite{Oel1988}.  In Proposition~\ref{prop_corrector} below, assuming the existence of a square integrable stream matrix, we prove that there exists a unique stationary gradient $\nabla\phi_i$ satisfying \eqref{intro_corrector}.  The quenched uniqueness of correctors under the higher $L^{d\vee(2+\d)}$-integrability assumption then follows from the Liouville theorem and Proposition~\ref{prop_sublinear} below.

Our first result is the quenched homogenization of \eqref{intro_transform} under the assumptions of uniform ellipticity and the existence of a $L^{d\vee(2+\d)}$-integrable stream matrix:
\begin{equation}\label{steady}  \textrm{Assume}\;\eqref{intro_rvs},\; \eqref{intro_ergodic},\; \eqref{intro_ue},\; \eqref{intro_stream},\; \textrm{and}\; \eqref{intro_gauge}.\end{equation}
This result provides a new approach to the results of \cite{Oel1988} and \cite{AveMaj1991}.  In comparison to \cite{KozTot2017,Tot2018}, which respectively assume the existence of an $L^2$/$L^{2+\d}$-integrable stream matrix to prove an annealed/quenched central limit theorem for the analogous random walk on the whole space, the methods of this paper explain how the higher $L^{d\vee(2+\d)}$-integrability assumption arises naturally to prove the strong convergence of the two-scale expansion in $H^1$.  This is done by introducing the homogenization flux correctors $\sigma_i$ satisfying
\begin{equation}\label{intro_flux}\nabla\cdot \sigma_i = (a+s)(\nabla\phi_i+e_i)-\overline{a}e_i.\end{equation}
The flux correction was used originally in the context of stochastic homogenization by Gloria, Neukamm, and Otto \cite{GloNeuOtt2020}, and allows the residuum of the two-scale expansion to be written in the form
 \begin{equation}\label{intro_00000}-\nabla\cdot (a^\ve+s^\ve)\nabla (u^\ve-v-\ve\phi^\ve_i\partial_i v) = \nabla\cdot\left[\left(\ve\phi_i^\ve(a^\ve+s^\ve)-\ve\sigma_i^\ve\right)\nabla(\partial_i v)\right].\end{equation}
The $L^{d\vee(2+\d)}$-integrability of $S$ is exactly the threshold which guarantees that the righthand side of \eqref{intro_00000} vanishes strongly in $L^2$ as $\ve\rightarrow 0$, using the sublinearity of $\phi_i$ and $\sigma_i$ proven in Proposition~\ref{prop_sublinear} below.  These methods apply without change to the elliptic and parabolic settings, and thereby establish an invariance principle on the whole space while also characterizing asymptotically the exit times and exit distributions of \eqref{intro_sde} from large domains.

\begin{thm*}[cf.\ Theorem~\ref{thm_ts} below]  Assume \eqref{steady}.  For some $\a\in(0,1)$ let $U\subseteq\R^d$ be a bounded $\C^{2,\a}$-domain, let $f\in\C^\a(U)$, and let $g\in \C^{2,\a}(\partial U)$.  For each $\ve\in(0,1)$ let $u^\ve\in H^1(U)$ be the unique solution of \eqref{intro_eq} and let $v\in H^1(U)$ be the unique solution of \eqref{intro_homogenized}.  Then, almost surely as $\ve\rightarrow 0$,
\[\lim_{\ve\rightarrow 0}\norm{u^\ve-v-\ve\phi^\ve_i\partial_iv}_{H^1(U)}=0.\]
\end{thm*}

The second main result of this work is an almost sure large-scale $\a$-H\"older regularity estimate for whole space solutions $u\in H^1_{\textrm{loc}}(\R^d)$ of the equation
\begin{equation}\label{intro_whole_space}-\nabla\cdot(a+s)\nabla u=0\;\;\textrm{in}\;\;\R^d.\end{equation}
Following \cite{GloNeuOtt2020}, based on the equivalence of Morrey, Campanato, and H\"older spaces (cf.\ eg.\ Giaquinta and Martinazzi \cite{GiaMar2012}), we introduce a version of the large-scale $\a$-H\"older semi-norm defined with respect to the intrinsic $(a+s)$-harmonic coordinates $(x_i+\phi_i)$:  the excess $\Exc(u;R)$ is defined by
\begin{equation}\label{intro_excess}\Exc(u;R)=\inf_{\xi\in\R^d}\left(R^{-2\a}\fint_{B_R}\abs{\nabla u-\xi-\nabla\phi_\xi}^2\right),\end{equation}
for $\phi_\xi=\xi_i\phi_i$.  The following theorem proves that there exists an almost surely finite radius $R_0\in(0,\infty)$ after which point the solutions of \eqref{intro_whole_space} enter the regime of $\a$-H\"older regularity.  The radius $R_0$ is quantified precisely by the sublinearity of the correctors in Proposition~\ref{prop_excess} below.

\begin{thm*}[cf.\ Proposition~\ref{prop_excess}, Theorem~\ref{thm_excess} below] Assume \eqref{steady}.  On a subset of full probability, for every $\a\in(0,1)$ there exists a random radius $R_0\in(0,\infty)$ and a deterministic $c\in(0,\infty)$ depending on $\a$ such that, for every weak solution $u\in H^1_{\textrm{loc}}(\R^d)$ of \eqref{intro_whole_space}, for every $R_1<R_2\in(R_0,\infty)$,
\[R^{-2\a}_1\Exc(u;R_1)\leq cR_2^{-2\a}\Exc(u;R_2),\]
for the excess $\Exc(u;R)$ defined in \eqref{intro_excess}.  \end{thm*}

The final result of this work is a first-order Liouville theorem.  In analogy with the classical first-order Liouville theorem, the $(a+s)$-harmonic coordinates $(x_i+\phi_i)$ are the linear functions in the random geometry of the space, and every subquadratic $(a+s)$-harmonic function is a corrector.  The sublinearity is quantified with respect to the $L^{2_*}$-norm, for $2_*>2$ defined below, as opposed to the $L^2$-norm used in \cite{GloNeuOtt2020}.  This stronger condition is necessary for our arguments due to the unboundedness of the stream matrix.  In combination, the Liouville theorem and Proposition~\ref{prop_sublinear} below prove the quenched uniqueness of the homogenization correctors and thereby provide a strong answer to the original question of \cite{AveMaj1991} on the physical space.

\begin{thm*}[cf.\ Theorem~\ref{thm_lvl} below]  Assume \eqref{steady}, let $q_d=d\vee(2+\d)$, and let $\nicefrac{1}{2_*}=\nicefrac{1}{2}-\nicefrac{1}{q_d}$.  Then, on a subset of full probability, every weak solution $u\in H^1_{\textrm{loc}}(\R^d)$ of \eqref{intro_whole_space} that is strictly subquadratic in the sense that, for some $\a\in(0,1)$,
\[\lim_{R\rightarrow\infty}\frac{1}{R^{1+\a}}\left(\fint_{B_R}\abs{u}^{2_*}\right)^\frac{1}{2_*}=0,\]
satisfies $u=c+\xi\cdot x+\phi_\xi$ in $H^1_{\textrm{loc}}(\R^d)$ for some $c\in\R$ and $\xi\in\R^d$.
\end{thm*}

\subsection{The organization of the paper}  In Section~\ref{sec_prob}, we describe how equations \eqref{intro_stream}, \eqref{intro_gauge}, and \eqref{intro_corrector} are lifted to the probability space:  in Section~\ref{sec_hom_cor} we construct the homogenization correctors, in Section~\ref{sec_flux_cor} we construct the homogenization flux correctors, and in Section~\ref{sec_stream} we prove the existence of a stationary stream matrix.  The proof of quenched homogenization is presented in Section~\ref{sec_stoch_hom}, where we also prove the well-posedness of \eqref{intro_eq} and the uniform ellipticity of the homogenized coefficient field.  In Section~\ref{sec_lsr}, we first obtain an energy estimate for the homogenization error in Proposition~\ref{prop_hom_energy} and then prove the large-scale regularity estimate.  We prove the Liouville theorem in Section~\ref{sec_lvl}, which is a consequence of the large-scale regularity estimate and a version of the Caccioppoli inequality adapted to the divergence-free setting.

\subsection{Overview of the literature}  The foundational theory of homogenization for elliptic and parabolic equations with periodic coefficients can be found in the references Bensoussan, Lions, and Papanicolaou \cite{BenLioPap2011} and Jikov, Kozlov, and Ole\u{\i}nik \cite{JikKozOle1994}.  The stochastic homogenization of divergence-form equations, and non-divergence-form equations without drift, was initiated by Papanicolaou and Varadhan \cite{PapVar1981,PapVar1982}, Osada \cite{Osa1983}, and Kozlov \cite{Koz1985}.  In the absence of additional assumptions on the drift, the general question of stochastic homogenization for diffusion equations of the type
\begin{equation}\label{lit_eq}-\nabla\cdot a(\nicefrac{x}{\ve},\o)\nabla u^\ve + \frac{1}{\ve}b(\nicefrac{x}{\ve},\o)\cdot\nabla u^\ve = f\;\;\textrm{in}\;\;U\;\;\textrm{with}\;\;u^\ve =g \;\;\textrm{on}\;\;\partial U\end{equation}
remains open.  The difficulty lies in constructing the invariant measure for the process from the point of view of the particle.  Thus far, the construction of this measure has required additional assumptions on the drift, such as the case when $b=\nabla U$ is the gradient of a stationary field, which has been treated, for instance, by Olla in \cite{Oll1994}, and the case when $b$ is divergence-free, which will be discussed in detail below.  The only other known results apply to a perturbative, strongly mixing, and isotropic regime in $d\geq 3$, which have been obtained in the discrete case by Bricmont and Kupiainen \cite{BriKup1991}, Bolthausen and Zeitouni \cite{BolZei2007}, Baur and Bolthausen \cite{BauBol2015}, and Baur \cite{Bau2016} and in the continuous case by Sznitman and Zeitouni \cite{SznZei2006} and the author \cite{Feh2017, Feh2017a,Feh2019} where \cite{Feh2017a} constructs the invariant measure.  A general overview can be found in the reference \cite{Oll1994} and the book Komorowski, Landim, and Olla \cite{KomLanOll2012}.

The homogenization of \eqref{lit_eq} with divergence-free drift was initiated by \cite{Osa1983}, who considered the case of a bounded stream matrix, and Oelschl\"ager \cite{Oel1988}, who proved an annealed invariance principle and the annealed homogenization of equations like \eqref{intro_eq} on the whole space assuming the existence of an $L^2$-integrable, $\C^2$-smooth stream matrix.  More recently, in the discrete case, T\'oth and Kozma \cite{KozTot2017} have proven an annealed invariance principle for the analogous discrete random walk under the so-called $\mathcal{H}_{-1}$-condition, which is equivalent to the existence of a stationary, $L^2$-integrable stream matrix.  T\'oth \cite{Tot2018} then proved a quenched central limit theorem in this setting assuming the existence of an $L^{2+\d}$-integrable stream matrix, using an adaptation of Nash's moment bound.  The higher $L^{d\vee(2+\d)}$-integrability assumption was introduced in Avellaneda and Majda \cite{AveMaj1991} to prove the quenched homogenization of the parabolic version of \eqref{intro_eq} on the whole space.  In \cite{AveMaj1991} correctors are constructed by approximation, and therefore lack an intrinsic characterization.  Related problems under more restrictive integrability assumptions have been considered by Fannjiang and Komorowski \cite{FanKom1997}, and time-dependent problems have been considered by Landim, Olla, and Yau \cite{LanOllYau1998}, Fannjiang and Komorowski \cite{FanKom1999,FanKom2002}, and Komorowski and Olla \cite{KomOll2001}.  Komorowski and Olla \cite{KomOll2002} have provided a counterexample to the annealed homogenization of equations like \eqref{intro_eq} on the whole space for drifts that do not admit a square-integrable stream matrix.

The relationship between Schauder estimates and Liouville theorems for constant-coefficient elliptic equations without drift was shown by Simon \cite{Sim1997}.  In the context of the periodic homogenization of divergence-form elliptic equations without drift Avellaneda and Lin \cite{AveLin1989} obtained a full hierarchy of Liouville theorems based on the large-scale regularity theory in H\"older- and $L^p$-spaces of the same authors Avellaneda and Lin \cite{AveLin1987,AveLin1987a}.   Armstrong and Smart \cite{ArmSma2016} first adapted the approach of \cite{AveLin1987a} to the stochastic case and obtained a large-scale regularity theory for environments satisfying a finite range of dependence.  Their methods are based on the variational characterization of solutions and quantify the convergence of certain sub- and super-additive energies.   Armstrong and Mourrat \cite{ArmMou2016} extended the results of \cite{ArmSma2016} to more general mixing conditions and these works have given rise to a significant literature on the subject.  A complete account of these developments can be found in the monograph Armstrong, Kuusi and Mourrat \cite{ArmKuuMou2019}, which includes applications to percolation clusters Armstrong and Dario \cite{ArmDar2018} and time-dependent environments Armstrong, Bordas, and Mourrat \cite{ArmBorMou2018}.

The results of this work are most closely related to those of Gloria, Neukamm and Otto \cite{GloNeuOtt2020}, who established a large-scale regularity theory and first-order Liouville principle for divergence-form equations without drift under the qualitative assumption of ergodicity.   In particular, the homogenization flux-correction introduced in \cite{GloNeuOtt2020} is used essentially in the proof of every result in this work, and their introduction of an intrinsic excess decay is used to obtain the large-scale regularity estimate and Liouville theorem.  Marahrens and Otto \cite{MarOtt2015} had previously obtained a Liouville theorem assuming a quantified form of ergodicity.  Fischer and Otto \cite{FisOtt2016, FisOtt2017} extended the results of \cite{GloNeuOtt2020}  to obtain a full hierarchy of Liouville theorems under a mild quantification of ergodicity.  Degenerate environments were considered by Bella, the author, and Otto \cite{BelFehOtt2018} and time-dependent environments by Bella, Chiarini, and the author \cite{BelChiFeh2019}.  The work \cite{GloNeuOtt2020} has similarly given rise to a substantial literature on the subject including, for instance, Gloria and Otto \cite{GloOtt2017} and Duerinckx, Gloria, and Otto \cite{DueGloOtt2020}.

\section{The extended homogenization corrector}\label{sec_prob}

In this section, we will describe how the equations \eqref{intro_stream}, \eqref{intro_gauge}, and \eqref{intro_corrector} are lifted to the probability space.  Following \cite{PapVar1981}, the transformation group $\{\tau_x\}_{x\in\R^d}$ is used to define so-called horizontal derivatives $\{D_i\}_{i\in\{1,\ldots d\}}$:  for each $i\in\{1,\ldots,d\}$,
\[\mathcal{D}(D_i)=\{f\in L^2(\O)\colon \lim_{h\rightarrow 0}\nicefrac{f(\tau_{he_i}\o)-f(\o)}{h}\;\textrm{exists strongly in}\;L^2(\O)\},\]
and $D_i\colon \mathcal{D}(D_i)\rightarrow L^2(\O)$ is defined by $D_if=\lim_{h\rightarrow 0}\nicefrac{f(\tau_{he_i}\o)-f(\o)}{h}.$
The $D_i$ are closed, densely defined operators on $L^2(\O)$.  We define $\mathcal{H}^1(\O)=\cap_{i=1}^d\mathcal{D}(D_i)$ and we will write $\mathcal{H}^{-1}(\O)$ for the dual of $\mathcal{H}^1(\O)$.  For $\phi\in \mathcal{H}^1(\O)$ we will write $D\phi=(D_1\phi,\ldots,D_d\phi)$ for the horizontal gradient.

A natural class of test functions can be constructed by convolution.  For each $\psi\in\C^\infty_c(\R^d)$ and $f\in L^\infty(\O)$ we define $\psi_f\in L^\infty(\O)$ as the convolution
\[\psi_f(\o)=\int_{\R^d}f(\tau_x\o)\psi(x)\dx,\]
and we will write $\mathcal{D}(\O)$ for the space of all such functions.  The space $\mathcal{D}(\O)$ is dense in $L^p(\O)$ for every $p\in[1,\infty)$.  We will write $\mathcal{D}'(\O)$ for the dual of $\mathcal{D}(\O)$, and we will understand distributional inequalities in $\mathcal{D}'(\O)$ in the sense that, for $f\in L^1(\O)$,
\[D_if=0\;\;\textrm{if and only if}\;\;\E[fD_i\psi]=0\;\;\textrm{for every}\;\;\psi\in\mathcal{D}(\O).\]
For a vector field $V=(V_i)_{i\in\{1,\ldots,d\}}\in L^2(\O;\R^d)$ we define the distributional divergence $D\cdot V=D_iV_i$.  The space of vector fields $L^2(\O;\R^d)$ then admits the following Helmoltz decomposition.  The space of potential or curl-free fields on $\O$ is defined by
\[L^2_{\textrm{pot}}(\O)=\overline{\left\{D\psi\in L^2(\O;\R^d)\colon \psi\in \mathcal{H}^1(\O)\right\}}^{L^2(\O;\R^d)},\]
which is the $L^2(\O;\R^d)$-closure of the space of $\mathcal{H}^1$-gradients.  The space of solenoidal or divergence-free fields is defined by
\[L^2_{\textrm{sol}}(\O)=\{V\in L^2(\O;\R^d)\colon D\cdot V=0\}.\]
The space $L^2(\O;\R^d)$ then admits the orthogonal decomposition
\[L^2(\O;\R^d)=L^2_{\textrm{pot}}(\O)\oplus L^2_{\textrm{sol}}(\O),\]
which can be deduced from Proposition~\ref{prop_flux} below.  We will now use this framework to lift equations like \eqref{intro_stream}, \eqref{intro_gauge}, and \eqref{intro_corrector} to the probability space.

\subsection{The homogenization corrector}\label{sec_hom_cor}  We will construct the homogenization corrector as a stationary gradient $\Phi_i$ in $L^2_{\textrm{pot}}(\O)$ satisfying
\begin{equation}\label{re_1}-D\cdot(A+S)(\Phi_i+e_i)=0\;\;\textrm{in}\;\;\mathcal{D}'(\O).\end{equation}
The solution is identified by approximation.  We first prove that for every $\a\in(0,1)$ there exists a unique $\Phi_{i,\a}\in \mathcal{H}^1(\O)$ satisfying the equation
\begin{equation}\label{cor_approx_cor_eq}\a\Phi_{i,\a}-D\cdot (A+S)(D\Phi_{i,\a}+e_i)=0,\end{equation}
where here, in comparison to \eqref{re_1}, the proof of uniqueness is simpler and relies crucially on the stationarity of $\Phi_{i,\a}$ itself.  We will then show that the $D\Phi_{i,\a}$ converge along the full sequence $\a\rightarrow 0$ in $L^2_{\textrm{pot}}(\O)$ to the unique solution of \eqref{re_1}.

The subsection is organized as follows.  We will first present a general proof of sublinearity for the homogenization correctors in Proposition~\ref{prop_sublinear} below.  We analyze \eqref{cor_approx_cor_eq} in Proposition~\ref{prop_approx_cor} below.  Finally,  in Proposition~\ref{prop_corrector} below, we prove that there exists a unique stationary gradient satisfying \eqref{re_1}.  The proof of Proposition~\ref{prop_corrector} is strongly motivated by \cite[Lemma~3.27]{Oel1988} and extends \cite[Lemma~3.27]{Oel1988} to the case of a general $L^2$-integrable stream matrix.  The proof of sublinearity is essentially well-known, but we include details here, in particular, to handle the less standard case $q=p_*$ below.

\begin{prop}\label{prop_sublinear}  Assume \eqref{steady}.  Let $p\in(1,\infty)$, let $F\in L^p(\O;\R^d)$ satisfy
\[D_iF_j=D_jF_i\;\;\textrm{for every}\;\;i,j\in\{1,\ldots,d\}\;\;\textrm{and}\;\;\E\left[F\right]=0,\]
and let $\phi\colon\R^d\times\O\rightarrow \R$ almost surely satisfy $\phi\in W^{1,p}_{\textrm{loc}}(\R^d)$ with $\nabla\phi(x,\omega)=F(\tau_x\omega)$.  If $p<d$ and $\nicefrac{1}{p_*}=\nicefrac{1}{p}-\nicefrac{1}{d}$ we have almost surely that, for every $q\in[1,p_*]$,
\begin{equation}\label{sub_00000}\lim_{R\rightarrow\infty}\frac{1}{R}\left(\fint_{B_R}\abs{\phi}^q\right)^\frac{1}{q}=0.\end{equation}
If $p\geq d$ then \eqref{sub_00000} holds for every $q\in[1,\infty)$.  If $p>d$ then $\lim_{R\rightarrow\infty}\nicefrac{\norm{\phi}_{L^\infty(B_R)}}{R}=0$.
 \end{prop}

\begin{proof}  We will first consider the case $p<d$ and $q=p_*$.  After rescaling we observe that
\begin{equation}\label{sub_0000}\limsup_{R\rightarrow\infty}\frac{1}{R}\left(\fint_{B_R}\abs{\phi}^{p_*}\right)^\frac{1}{p_*}=\limsup_{\ve\rightarrow 0}\left(\fint_{B_1}\abs{\phi^\ve}^{p_*}\right)^\frac{1}{p_*},\end{equation}
for $\phi^\ve(x)=\ve\phi(\nicefrac{x}{\ve})$.  We will first prove that
\begin{equation}\label{sub_000}\limsup_{\ve\rightarrow 0}\left(\fint_{B_1}\abs{\phi^\ve-\fint_{B_1}\phi^\ve}^{p_*}\right)^\frac{1}{p_*}=0,\end{equation}
and then show that \eqref{sub_000} implies \eqref{sub_00000}.  For every $\d\in(0,1)$ let $\rho^\d$ be a standard convolution kernel of scale $\d$, and for every $\ve,\d\in(0,1)$ let $\phi^{\ve,\d}=\phi^\ve*\rho^\d$.  The triangle inequality proves that, for every $\ve,\d\in(0,1)$,
\begin{align*}
\left(\fint_{B_1}\abs{\phi^\ve-\fint_{B_1}\phi^\ve}^{p_*}\right)^\frac{1}{p_*} & \leq  \left(\fint_{B_1}\abs{\phi^{\ve,\d}-\fint_{B_1}\phi^{\ve,\d}}^{p_*}\right)^\frac{1}{p_*}
\\ \nonumber & \quad + \left(\fint_{B_1}\abs{\left(\phi^\ve-\fint_{B_1}\phi^\ve\right)-\left(\phi^{\ve,\d}-\fint_{B_1}\phi^{\ve,\d}\right)}^{p_*}\right)^\frac{1}{p_*}.
\end{align*}
The Sobolev inequality proves that there exists $c\in(0,\infty)$ such that, for each $\ve,\delta\in(0,1)$,
\begin{align}\label{sub_0}
\left(\fint_{B_1}\abs{\phi^\ve-\fint_{B_1}\phi^\ve}^{p_*}\right)^\frac{1}{p_*} & \leq  c\left(\left(\fint_{B_1}\abs{\nabla \phi^{\ve,\d}}^p\right)^\frac{1}{p}+ \left(\fint_{B_1}\abs{\nabla\left(\phi^\ve-\phi^{\ve,\d}\right)}^{p}\right)^\frac{1}{p}\right).
\end{align}
For the first term on the righthand side of \eqref{sub_0}, since Jensen's inequality proves that, for every $x\in B_1$ and $\ve,\d\in(0,1)$,
\[\abs{\nabla\phi^{\ve,\d}(x,\omega)}^p=\abs{\int_{\R^d}\nabla\phi^\ve(y)\rho^\delta(y-x)\dy}^p\leq\norm{\rho^\delta}_{L^\infty(\R^d)}\int_{B_2}\abs{\nabla\phi^\ve}^p,\]
the ergodic theorem and $F\in L^p(\O;\R^d)$ prove that, almost surely for each $\delta\in(0,1)$,
\begin{equation}\label{sub_3}\sup_{\ve\in(0,1)}\left(\sup_{x\in B_1}\abs{\nabla\phi^{\ve,\d}(x,\omega)}\right)<\infty.\end{equation}
Since the ergodic theorem and $\E[F]=0$ prove almost surely that, as $\ve\rightarrow 0$,
\[\nabla\phi^\ve(x,\omega)\rightharpoonup 0\;\;\textrm{weakly in}\;\;L^p_{\textrm{loc}}(\R^d;\R^d),\]
we have, almost surely for every $\delta\in(0,1)$ and $x\in B_1$,
\begin{equation}\label{sub_1}\lim_{\ve\rightarrow 0}\abs{\nabla \phi^{\ve,\d}(x)}^p=\lim_{\ve\rightarrow 0}\abs{\int_{\R^d}\nabla\phi^\ve(y)\rho^\d(y-x)\dy}^p=0.\end{equation}
The dominated convergence theorem, \eqref{sub_3}, and \eqref{sub_1} then prove that, almost surely for every $\delta\in(0,1)$,
\begin{equation}\label{sub_04}\limsup_{\ve\rightarrow 0}\left(\fint_{B_1}\abs{\nabla \phi^{\ve,\d}}^p\right)^\frac{1}{p}=0.\end{equation}
For the second term on the righthand side of \eqref{sub_0}, the ergodic theorem proves almost surely for every $\delta\in(0,1)$ that
\begin{equation}\label{sub_4}\lim_{\ve\rightarrow 0}\left(\fint_{B_1}\abs{\nabla\left(\phi^\ve-\phi^{\ve,\d}\right)}^{p}\right)^\frac{1}{p}=\E\left[\abs{F-F^\delta}^p\right]^\frac{1}{p},\end{equation}
for $F^\delta(\omega)=\int_{\R^d}F(\tau_y\omega)\rho^\delta(y)\dy$.  Returning to \eqref{sub_0}, it follows from \eqref{sub_04} and \eqref{sub_4} that, for every $\d\in(0,1)$,
\[\limsup_{\ve\rightarrow 0}\left(\fint_{B_1}\abs{\phi^\ve-\fint_{B_1}\phi^\ve}^{p_*}\right)^\frac{1}{p_*} \leq \E\left[\abs{F-F^\delta}^p\right]^\frac{1}{p}.\]
It follows from $F\in L^p(\O;\R^d)$ that $\lim_{\delta\rightarrow 0}\E\left[\abs{F-F^\delta}^p\right]^\frac{1}{p}=0$, which complete the proof of \eqref{sub_000}.

It remains to prove that \eqref{sub_000} implies \eqref{sub_00000}.  The following argument appears in \cite[Lemma~2]{BelFehOtt2018}.  Due to the equivalence \eqref{sub_0000}, almost surely for every $\delta\in(0,1)$ there exists $R_0\in(0,\infty)$ such that, for every $R\geq R_0$,
\[\left(\fint_{B_R}\abs{\phi-\fint_{B_R}\phi}^{p_*}\right)^\frac{1}{p_*}\leq R\delta.\]
By the triangle inequality, for every $R\in[R_0,2R_0]$, for $c\in(0,\infty)$ independent of $R$,
\begin{align*}
\abs{\fint_{B_R}\phi-\fint_{B_{R_0}}\phi} & \leq \left(\fint_{B_{R_0}}\abs{\phi-\fint_{B_R}\phi}^{p_*}\right)^\frac{1}{p_*}+ \left(\fint_{B_{R_0}}\abs{\phi-\fint_{B_{R_0}}\phi}^{p_*}\right)^\frac{1}{p_*}
\\ & \leq \left(\frac{R}{R_0}\right)^\frac{d}{p_*}R\delta+R_0\delta \leq \left(2^{(\nicefrac{d}{p_*}+1)}+1\right)R_0\delta=cR_0\delta.
\end{align*}
Therefore, for every $R\in[R_0,2R_0]$,
\[\abs{\frac{1}{R}\fint_{B_R}\phi}\leq \left(\frac{R_0}{R}\right)\abs{\frac{1}{R_0}\fint_{B_{R_0}}\phi}+c\left(\frac{R_0}{R}\right)\delta.\]
It then follows inductively that, for every $R\in[2^{k-1}R_0,2^kR_0]$,
\begin{align*}
\abs{\frac{1}{R}\fint_{B_R}\phi} & \leq \left(\frac{2^{k-1}R_0}{R}\right)\abs{\frac{1}{2^{k-1}R_0}\fint_{B_{2^{k-1}R_0}}\phi}+c\left(\frac{2^{k-1}R_0}{R}\right)\delta
\\ & \leq \abs{\frac{1}{R}\fint_{B_{R_0}}\phi} + c\left(\sum_{j=0}^\infty2^{-j}\right)\delta= \abs{\frac{1}{R}\fint_{B_{R_0}}\phi}+2c\delta.
\end{align*}
Since $\delta\in(0,1)$ was arbitrary, we have almost surely that
\begin{equation}\label{sub_5}\limsup_{R\rightarrow\infty} \abs{\frac{1}{R}\fint_{B_R}\phi}=0.\end{equation}
The triangle inequality, \eqref{sub_000}, and \eqref{sub_5} prove almost surely that
\[\limsup_{R\rightarrow\infty}\frac{1}{R}\left(\fint_{B_R}\abs{\phi}^{p_*}\right)^\frac{1}{p_*}\leq \limsup_{R\rightarrow\infty}\frac{1}{R}\left(\fint_{B_R}\abs{\phi-\int_{B_R}\phi}^{p_*}\right)^\frac{1}{p_*}+\limsup_{R\rightarrow\infty}\abs{\frac{1}{R}\fint_{B_R}\phi}=0,\]
which completes the proof for the case $p<d$ and $q=p_*$.  The fact that \eqref{sub_00000} holds for the cases $p<d$ and $q\in[1,p^*)$ and $p\geq d$ and $q\in[1,\infty)$ is then a consequence of H\"older's inequality.   If $p>d$, returning to \eqref{sub_000}, the Sobolev embedding theorem implies that the sequence $\{\phi^\ve-\int_{B_1}\phi^\ve\}_{\ve\in(0,1)}$ is almost surely bounded in $\C^\alpha(B_1)$ for $\alpha=1-\nicefrac{d}{p}$.  The Arzel\`a-Ascoli theorem, \eqref{sub_1}, and \eqref{sub_5} then prove almost surely that
\[\lim_{\ve\rightarrow 0}\norm{\phi^\ve}_{L^\infty(B_1)}=\lim_{R\rightarrow\infty}\frac{1}{R}\norm{\phi}_{L^\infty(B_R)}=0.\qedhere\]
\end{proof}

\begin{prop}\label{prop_approx_cor} Assume \eqref{steady} under the weaker assumption $S\in L^2(\O;\R^{d\times d})$.  Let $F\in L^2(\O;\R^d)$ and $\alpha\in(0,1)$.  Then there exists a unique $\Phi\in \mathcal{H}^1(\O)$ satisfying the equation
\[\a\Phi-D\cdot (A+S)D\Phi=-D\cdot F\;\;\textrm{in}\;\;D'(\O).\]
Furthermore, $\Phi$ satisfies the energy identity
\begin{equation}\label{cor_000}\E\left[\alpha \Phi^2+AD\Phi\cdot D\Phi\right]=\E\left[F\cdot D\Phi\right].\end{equation}
\end{prop}

\begin{proof}  We will write $S=(S_{ij})_{i,j\in\{1,\ldots,d\}}\in L^2(\O;\R^{d\times d})$ and for every $n\in\N$ we define
\[S_n=((S_{ij}\wedge n)\vee(-n))_{i,j\in\{1,\ldots,d\}}.\]
The Lax-Milgram theorem proves that there exists a unique $\Phi_n\in \mathcal{H}^1(\O)$ which satisfies
\[\a\Phi_n-D\cdot (A+S_n)D\Phi_n = -D\cdot F\;\;\textrm{in}\;\;\mathcal{H}^{-1}(\O).\]
The boundedness and anti-symmetry of $S_n$, the uniform ellipticity of $A$, H\"older's inequality, and Young's inequality prove that there exists $c\in(0,\infty)$ such that, for each $n\in\N$,
\begin{equation}\label{cor_1}\E\left[\a\Phi_n^2+\abs{D\Phi_n}^2\right]\leq c \E\left[\abs{F}^2\right].\end{equation}
It follows from \eqref{cor_1} that there exists $\Phi\in \mathcal{H}^1(\O)$ such that, after passing to a subsequence, as $n\rightarrow\infty$,
\begin{equation}\label{cor_2}\Phi_n\rightharpoonup\Phi\;\;\textrm{weakly in}\;\;\mathcal{H}^1(\O).\end{equation}
Since $S_n\rightarrow S$ strongly in $L^2(\O,\R^{d\times d})$, it follows from \eqref{cor_2} and $\mathcal{D}(\O)\subseteq \mathcal{H}^1(\O)$ that $\Phi$ solves
\begin{equation}\label{cor_3}\a\Phi-D\cdot(A+S)D\Phi=-D\cdot F\;\;\textrm{in}\;\;\mathcal{D}'(\O).\end{equation}
Uniqueness is an immediate consequence of the linearity, the uniform ellipticity of $A$, and the energy estimate.  Therefore, it remains only to prove the energy estimate \eqref{cor_000}.  The skew-symmetry of $S$ and $D\cdot(D\cdot S)=0$ prove that, for every $\psi\in D(\O)$,
\begin{equation}\label{cor_4}\E\left[SD\Phi\cdot D\psi\right]=-\E\left[(D\cdot S)\cdot D\psi \Phi\right]=\E\left[\left((D\cdot S)\cdot D\Phi\right)\psi\right].\end{equation}
For each $n\in \N$ let $\Phi_n=(\Phi\wedge n)\vee(-n)$ and for every $\ve\in(0,1)$ let $\rho^\ve$ denote a standard convolution kernel of scale $\ve\in(0,1)$.  For every $n\in\N$ and $\ve\in(0,1)$ let
\[\Phi_{n,\ve}(\o)=\int_{\R^d}\Phi_n(\tau_x\o)\rho^\ve(x)\dx.\]
It follows by definition that the $\Phi_{n,\ve}$ are admissible test functions for \eqref{cor_3} and, after using the boundedness of $\Phi_n$ to pass to the limit $\ve\rightarrow 0$, it follows from \eqref{cor_3} and \eqref{cor_4} that
\begin{equation}\label{cor_5}\E\left[\a\Phi\Phi_n+AD\Phi\cdot D\Phi_n\right]=-\E\left[(D\cdot S)D\Phi\Phi_n\right]+\E\left[F\cdot D\Phi_n\right].\end{equation}
Since the distributional equality $D\Phi_n=D\Phi\mathbf{1}_{\{\abs{\Phi}\leq n\}}$ proves the distributional equality $D\Phi\Phi_n=D\left(\Phi\Phi_n-\nicefrac{1}{2}\Phi_n^2\right)$, the boundedness of $\Phi_n$ and $D\cdot(D\cdot S)=0$ prove that
\begin{equation}\label{cor_6}\E\left[(D\cdot S)D\Phi\Phi_n\right]=\E\left[(D\cdot S)\cdot D\left(\Phi\Phi_n-\nicefrac{1}{2}\Phi_n^2\right)\right]=0.\end{equation}
It then follows from \eqref{cor_5}, \eqref{cor_6}, and $D\Phi_n=D\Phi\mathbf{1}_{\{\abs{\Phi}\leq n\}}$ that
\begin{equation}\label{cor_7}\E\left[\a\Phi\Phi_n+AD\Phi\cdot D\Phi\mathbf{1}_{\{\abs{\Phi}\leq n\}}\right]=\E\left[F\cdot D\Phi\textbf{1}_{\{\abs{\Phi}\leq n\}}\right].\end{equation}
The energy estimate \eqref{cor_000} then follows by the dominated convergence theorem, after passing to the limit $n\rightarrow\infty$ in \eqref{cor_7}.  This completes the proof.\end{proof}

\begin{prop}\label{prop_corrector}  Assume \eqref{steady} under the weaker assumption $S\in L^2(\O;\R^{d\times d})$.  Let $F\in L^2(\O;\R^d)$.  Then there exists a unique $\Phi\in L^2_{\textrm{pot}}(\O)$ which satisfies the equation
\begin{equation}\label{cor_0008}-D\cdot (A+S)\Phi=-D\cdot F\;\;\textrm{in}\;\;D'(\O).\end{equation}
Furthermore, $\Phi$ satisfies the energy identity
\begin{equation}\label{cor_008}\E\left[A\Phi\cdot\Phi\right]=\E\left[F\cdot \Phi\right].\end{equation}
\end{prop}

\begin{proof}  For every $\a\in(0,1)$ let $\Phi_\a\in \mathcal{H}^1(\O)$ be the unique solution of 
\begin{equation}\label{cor_08}\a\Phi_\a-D\cdot (A+S)D\Phi_\a=-D\cdot F\;\;\textrm{in}\;\;D'(\O).\end{equation}
It follows from \eqref{cor_000}, the uniform ellipticity, H\"older's inequality, and Young's inequality that, for some $c\in(0,\infty)$ independent of $\a\in(0,1)$,
\[\E\left[\a(\Phi_\a)^2+\abs{D\Phi_\a}^2\right]\leq c\E\left[\abs{F}^2\right].\]
Therefore, after passing to a subsequence $\a\rightarrow 0$, there exists $\Phi\in L^2_{\textrm{pot}}(\O)$ such that
\begin{equation}\label{cor_8}a\Phi_\a\rightarrow 0\;\;\textrm{strongly in}\;\;L^2(\O)\;\;\textrm{and}\;\;D\Phi_\a\rightharpoonup \Phi\;\;\textrm{weakly in}\;\;L^2_\textrm{pot}(\O).\end{equation}
It follows from \eqref{cor_08} and \eqref{cor_8} that
\begin{equation}\label{cor_9}-D\cdot (A+S)\Phi=-D\cdot F\;\;\textrm{in}\;\;D'(\O).\end{equation}
This completes the proof of existence.

We will now prove the energy identity \eqref{cor_008}.  Since $\Phi\in L^2_\textrm{pot}(\O)$ is curl-free and satisfies \eqref{cor_9}, by integration let $\phi\colon\R^d\times\O\rightarrow \R$ be the unique function almost surely satisfying $\fint_{B_1}\phi = 0$, $\phi\in H^1_\textrm{loc}(\R^d)$ with $\nabla\phi(x,\o)=\Phi(\tau_x\o)$, and, for every $\psi\in\C^\infty_c(\R^d)$,
\[\int_{\R^d}(a+s)\nabla\phi\cdot\nabla\psi =\int_{\R^d}f\cdot \nabla\psi,\]
for $a(x,\omega)=A(\tau_x\o)$, $s(x,\o)=S(\tau_x\o)$, and $f(x,\o)=F(\tau_x\o)$.   Since $\nabla\cdot(\nabla\cdot s)= 0$ almost surely on $\R^d$, a repetition of the argument from Proposition~\ref{prop_approx_cor} proves that, for every $\psi\in\C^\infty_c(\R^d)$,
\[\int_{\R^d}s\nabla\phi\cdot\nabla\psi = \int_{\R^d}(\nabla\cdot s)\cdot \nabla\phi \psi.\]
Therefore, almost surely for every $\psi\in\C^\infty_c(\R^d)$,
\begin{equation}\label{cor_11}\int_{\R^d}a\nabla\phi\cdot\nabla\psi+\left((\nabla\cdot s)\cdot\nabla\phi\right)\psi =\int_{\R^d}f\cdot \nabla\psi.\end{equation}
Let $\eta\colon\R^d\rightarrow[0,1]$ be a smooth function satisfying $\eta=1$ on $\overline{B}_1$ and $\eta=0$ on $\R^d\setminus B_2$, and for every $R\in(0,\infty)$ let $\eta_R(x)=\eta(\nicefrac{x}{R})$.  For every $\ve\in(0,1)$ let $\rho^\ve\in\C^\infty_c(\R^d)$ be a standard convolution kernel of scale $\ve\in(0,1)$.  For every $n\in\N$ let $\phi_n=(\phi\wedge n)\vee(-n)$.  Then for every $R\in(0,\infty)$, $\ve\in(0,1)$, and $n\in\N$, the function $(\phi_n*\rho^\ve)\eta_R$ is an admissible test function for \eqref{cor_11}.  Using the boundedness of $\phi_n$ to pass to the limit $\ve\rightarrow 0$, we have almost surely for every $n\in\N$ that
\begin{equation}\label{cor_12}\int_{\R^d}a\nabla\phi\cdot\nabla\phi_n\eta_R+a\nabla\phi\cdot \nabla\eta_R\phi_n +\left((\nabla\cdot s)\cdot\nabla\phi\right)\phi_n\eta_R =\int_{\R^d}f\cdot \nabla\phi_n\eta_R+f\cdot \nabla\eta_R\phi_n.\end{equation}
The distributional equality $\nabla\phi\phi_n=\nabla\left(\phi\phi_n-\nicefrac{1}{2}\phi_n^2\right)$, the fact that $(\nabla\cdot s)$ is divergence-free, and the skew-symmetry of $s$ prove that
\begin{equation}\label{cor_014}\int_{\R^d}\left((\nabla\cdot s)\cdot\nabla\phi\right)\phi_n\eta_R=-\int_{\R^d}\left((\nabla\cdot s)\cdot \nabla\eta_R\right)\left(\phi\phi_n-\nicefrac{1}{2}\phi_n^2\right)=\int_{\R^d}\left(s\nabla\phi\cdot \nabla \eta_R\right)\phi_n.\end{equation}
It then follows from \eqref{cor_12}, \eqref{cor_014}, and the distributional equality $\nabla\phi_n=\nabla\phi\mathbf{1}_{\{\abs{\phi}\leq n\}}$ that
\begin{align}\label{cor_14}
& \int_{\R^d}a\nabla\phi\cdot\nabla\phi\mathbf{1}_{\{\abs{\phi}\leq n\}}\eta_R-\int_{\R^d}f\cdot \nabla\phi\mathbf{1}_{\{\abs{\phi}\leq n\}}\eta_R
\\ \nonumber & =\int_{\R^d}f\cdot \nabla\eta_R\phi_n- a\nabla\phi\cdot \nabla\eta_R\phi_n -\left(s\nabla\phi\cdot \nabla \eta_R\right)\phi_n.
\end{align}
For each $R\in(0,\infty)$ let $c_R=\int_{\R^d}\eta_R$.  It follows from the definition of $\eta_R$ that $\abs{B_R}\leq c_R\leq \abs{B_{2R}}$.  We now make the choice $n=R$.  It then follows from the definition of $\eta_R$, the definition of $c_R$, the definition of $\phi_n$, the uniform ellipticity, \eqref{cor_14}, and H\"older's inequality that, for some $c\in(0,\infty)$ independent of $R$,
\begin{align}\label{cor_15}
& \abs{c_R^{-1}\int_{\R^d}a\nabla\phi\cdot\nabla\phi\mathbf{1}_{\{\abs{\phi}\leq R\}}\eta_R-c_R^{-1}\int_{\R^d}f\cdot \nabla\phi\mathbf{1}_{\{\abs{\phi}\leq R\}}\eta_R}
\\ \nonumber & \leq c\left(\frac{1}{R}\left(\fint_{B_{2R}}\abs{\phi}^2\right)^\frac{1}{2}\left(\fint_{B_{2R}}\abs{f}^2+\abs{\nabla\phi}^2\right)^\frac{1}{2}+\frac{1}{R}\left(\fint_{B_{2R}}\abs{s}^2\right)^\frac{1}{2}\left(\fint_{B_{2R}}\abs{\nabla\phi}^2\phi_R^2\right)^\frac{1}{2}\right).
\end{align}
The difficulty in the proof is that, since $\phi$ is not itself stationary, it is not obvious for instance that
\begin{equation}\label{cor_16}\lim_{R\rightarrow\infty}c_R^{-1}\int_{\R^d}a\nabla\phi\cdot\nabla\phi\mathbf{1}_{\{\abs{\phi}\leq R\}}\eta_R =\E\left[a\Phi\cdot\Phi\right],\end{equation}
as is formally suggested by the ergodic theorem.  We will prove \eqref{cor_16} using the sublinearity of $\phi$.  For each $R\in(0,\infty)$ we write
\[c_R^{-1}\int_{\R^d}a\nabla\phi\cdot\nabla\phi\mathbf{1}_{\{\abs{\phi}\leq R\}}\eta_R= c_R^{-1}\int_{\R^d}a\nabla\phi\cdot\nabla\phi\eta_R-c_R^{-1}\int_{\R^d}a\nabla\phi\cdot\nabla\phi\mathbf{1}_{\{\abs{\phi}>R\}}\eta_R.\]
Since the ergodic theorem proves almost surely that
\[\E\left[A\Phi\cdot\Phi\right] =\lim_{R\rightarrow\infty}c_R^{-1}\int_{\R^d}a\nabla\phi\cdot\nabla\phi\eta_R,\]
it remains only to prove almost surely that
\begin{equation}\label{cor_18}\lim_{R\rightarrow\infty}c_R^{-1}\int_{\R^d}a\nabla\phi\cdot\nabla\phi\mathbf{1}_{\{\abs{\phi}>R\}}\eta_R=0.\end{equation}
Chebyshev's inequality and the definition of $\eta_R$ prove that, for $c\in(0,\infty)$ independent of $R\in(0,\infty)$,
\[c_R^{-1}\abs{\{\abs{\phi}> R\}\cap \Supp(\eta_R)}\leq c_R^{-1}\abs{\{\abs{\phi}> R\}\cap B_{2R}}\leq \frac{c}{R^2}\fint_{B_{2R}}\abs{\phi}^2,\]
from which we almost surely conclude using Proposition~\ref{prop_sublinear} that
\begin{equation}\label{cor_19}\limsup_{R\rightarrow\infty}c_R^{-1}\abs{\{\abs{\phi}> R\}\cap \Supp(\eta_R)}\leq \limsup_{R\rightarrow\infty}\frac{c}{R^2}\fint_{B_{2R}}\abs{\phi}^2=0.\end{equation}
We now exploit the stationarity of $\nabla\phi$.  For each $R\in(0,\infty)$ and $K\in\N$, the uniform ellipticity and the definitions of $\eta_R$ and $c_R$ prove that, for some $c\in(0,\infty)$ independent of $R$ and $K$,
\begin{align}\label{cor_0000020}
\abs{c_R^{-1}\int_{\R^d}a\nabla\phi\cdot\nabla\phi\mathbf{1}_{\{\abs{\phi}>R\}}\eta_R} & \leq \abs{c_R^{-1}\int_{\R^d}a\nabla\phi\cdot\nabla\phi\mathbf{1}_{\{\abs{\phi}>R,\abs{\nabla\phi}\leq K\}}\eta_R}
\\ \nonumber &  \quad +c\abs{\fint_{B_{2R}}\abs{\nabla\phi}^2\mathbf{1}_{\{\abs{\nabla\phi}>K\}}}.
\end{align}
After applying the dominated convergence theorem, the ergodic theorem, the stationarity of $\nabla\phi$, and \eqref{cor_19} to \eqref{cor_0000020}, we have almost surely for every $K\in\N$ that
\[\limsup_{R\rightarrow\infty}\abs{c_R^{-1}\int_{\R^d}a\nabla\phi\cdot\nabla\phi\mathbf{1}_{\{\abs{\phi}>R\}}\eta_R} \leq c\E\left[\abs{\Phi}^2\mathbf{1}_{\{\abs{\Phi}>K\}}\right].\]
Therefore, since $\Phi\in L^2_{\textrm{pot}}(\O)$, after passing to the limit $K\rightarrow\infty$ we conclude the proof of \eqref{cor_18} and therefore the proof of \eqref{cor_16}.  The identical proof shows almost surely that
\begin{equation}\label{cor_020}\lim_{R\rightarrow\infty} c_R^{-1}\int_{\R^d}f\cdot\nabla\phi\mathbf{1}_{\{\abs{\phi}\leq R\}}=\E\left[F\cdot\Phi\right].\end{equation}
It remains to treat the two terms on the righthand side of \eqref{cor_15}.  Proposition~\ref{prop_sublinear}, $F,\Phi\in L^2(\O;\R^d)$, and the ergodic theorem prove almost surely that
\begin{equation}\label{cor_20}\lim_{R\rightarrow\infty}\left[\frac{1}{R}\left(\fint_{B_{2R}}\abs{\phi}^2\right)^\frac{1}{2}\left(\fint_{B_{2R}}\abs{f}^2+\abs{\nabla\phi}^2\right)^\frac{1}{2}\right]=0.\end{equation}
For the final term on the righthand side of \eqref{cor_15}, it follows from the definition of $\phi_R$ and the triangle inequality that, for each $K\in\N$ and $R\in(0,\infty)$,
\begin{align*}
& \frac{1}{R}\left(\fint_{B_{2R}}\abs{s}^2\right)^\frac{1}{2}\left(\fint_{B_{2R}}\abs{\nabla\phi}^2\phi_R^2\right)^\frac{1}{2}
\\ & \leq \frac{1}{R}\left(\fint_{B_{2R}}\abs{s}^2\right)^\frac{1}{2}\left(\fint_{B_{2R}\cap\{\abs{\nabla\phi}\leq K\}}\abs{\nabla\phi}^2\phi_R^2\right)^\frac{1}{2}
\\ & \quad+ \frac{1}{R}\left(\fint_{B_{2R}}\abs{s}^2\right)^\frac{1}{2}\left(\fint_{B_{2R}\cap\{\abs{\nabla\phi}>K\}}\abs{\nabla\phi}^2\phi_R^2\right)^\frac{1}{2}
\\ & \leq \frac{K}{R}\left(\fint_{B_{2R}}\abs{s}^2\right)^\frac{1}{2}\left(\fint_{B_{2R}}\phi^2\right)^\frac{1}{2} +\left(\fint_{B_{2R}}\abs{s}^2\right)^\frac{1}{2}\left(\fint_{B_{2R}\cap\{\abs{\nabla\phi}>K\}}\abs{\nabla\phi}^2\right)^\frac{1}{2}.
\end{align*}
The ergodic theorem, the stationarity of $\nabla\phi$, $S\in L^2(\O;\R^{d\times d})$, and Proposition~\ref{prop_sublinear} then prove almost surely that, for some $c\in(0,\infty)$ independent of $K\in\N$,
\[\limsup_{R\rightarrow\infty}\frac{1}{R}\left(\fint_{B_{2R}}\abs{s}^2\right)^\frac{1}{2}\left(\fint_{B_{2R}}\abs{\nabla\phi}^2\phi_R^2\right)^\frac{1}{2}\leq c\E\left[\abs{\Phi}^2\mathbf{1}_{\{\abs{\Phi}>K\}}\right].\]
Since $\Phi\in L^2_{\textrm{pot}}(\O)$, after passing to the limit $K\rightarrow\infty$ we conclude almost surely that
\begin{equation}\label{cor_22}\limsup_{R\rightarrow\infty}\frac{1}{R}\left(\fint_{B_{2R}}\abs{s}^2\right)^\frac{1}{2}\left(\fint_{B_{2R}}\abs{\nabla\phi}^2\phi_R^2\right)^\frac{1}{2}=0.\end{equation}
In combination \eqref{cor_15}, \eqref{cor_16}, \eqref{cor_020}, \eqref{cor_20}, and \eqref{cor_22} prove that
\begin{equation}\label{cor_23}\E\left[A\Phi\cdot\Phi\right]=\E\left[F\cdot\Phi\right],\end{equation}
which complete the proof of the energy identity.

It remains only to prove uniqueness.  Suppose that $\Phi_1, \Phi_2\in L^2_{\textrm{pot}}(\O)$ satisfy \eqref{cor_0008} and \eqref{cor_008}.  Then by linearity the difference $\Phi_1-\Phi_2$ satisfies both \eqref{cor_0008} and \eqref{cor_008} with $F=0$.  The uniform ellipticity and the energy identity \eqref{cor_23} prove that
\[\lambda \E\left[\abs{\Phi_1-\Phi_2}^2\right]\leq \E\left[A\cdot(\Phi_1-\Phi_2)\cdot(\Phi_1-\Phi_2)\right]=0,\]
which proves that $\Phi_1=\Phi_2$ in $L^2_\textrm{pot}(\O)$ and completes the proof.  \end{proof}

\subsection{The homogenization flux corrector}\label{sec_flux_cor}  We will now construct the skew-symmetric flux correctors $\sigma_i$ satisfying \eqref{intro_flux}.  Let $p_d\in(1,2)$ denote the integrability exponent
\begin{equation}\label{pd_exponent}p_d=\frac{2d}{d+2}\;\;\textrm{if}\;\;d\geq 3\;\;\textrm{and}\;\;p_d=\frac{4+2\d}{4+\d}\;\;\textrm{if}\;\;d=2,\end{equation}
and for each $i\in\{1,\ldots,d\}$, using H\"older's inequality, let $Q_i\in L^{p_d}(\O;\R^d)$ be befined by
\[Q_i=(A+S)(\Phi_i+e_i),\]
for the corrector fields $\Phi_i\in L^2_{\textrm{pot}}(\O)$ satisfying \eqref{re_1} constructed in Proposition~\ref{prop_corrector}.  We will identify the flux correctors $\sigma_i=(\sigma_{ijk})$ by their stationary gradients $\Sigma_{ijk}$ satisfying the equation
\[-D\cdot \Sigma_{ijk}=D_jQ_{ik}-D_k Q_{ij}\;\;\textrm{in}\;\;\mathcal{D}'(\O).\]
We construct the $\Sigma_{ijk}$ in Proposition~\ref{prop_flux} below.  In Proposition~\ref{prop_flux_corrector} below, we prove that the resulting skew-symmetric matrices $\sigma_i$ defined on $\R^d$ by integration almost surely satisfy $\nabla\cdot\sigma_i= q_i$ for $q_i(x,\o)=Q_i(\tau_x\o)-\E[Q_i]$.  The proof of existence and uniqueness for the flux correctors is an extension of \cite[Lemma~1]{BelFehOtt2018}.
\begin{prop}\label{prop_flux}  Assume \eqref{steady}.  Let $p\in(1,\infty)$ and let $F\in L^p(\O;\R^d)$.  Then there exists a unique weak solution $\Phi\in L^p(\O;\R^d)$ of the equation
\begin{equation}\label{pf_1}-D\cdot \Phi= -D\cdot F\;\;\textrm{in}\;\;\mathcal{D}'(\O),\end{equation}
with $\E[\Phi]=0$ and such that, for every $i,j\in\{1,2,\ldots,d\}$,
\begin{equation}\label{pf_2}D_i\Phi_j=D_j\Phi_i\;\;\textrm{in}\;\;\mathcal{D}'(\O).\end{equation}
\end{prop}

\begin{proof}  Let $p\in(1,\infty)$.  We will first consider a smooth righthand side $F=(F_1,\ldots,F_d)\in \mathcal{D}(\O)^d$.  The Lax-Milgram theorem proves that, for every $\alpha\in(0,1)$, there exists a unique solution $\Phi_\a\in L^2_{\textrm{pot}}(\O)$ of the equation
\[\a\Phi_\a-D\cdot D\Phi_\a= -D\cdot F\;\;\textrm{in}\;\;\mathcal{D}'(\O).\]
Since $\Phi_\a\in L^2_{\textrm{pot}}(\O)$ is mean zero and curl-free, define by integration $\phi_\a\colon\R^d\times\O\rightarrow\R$ which almost surely satisfies $\phi_\a\in H^1_{\textrm{loc}}(\R^d)$, $\nabla\phi_\a(x,\o)=\Phi_\a(\tau_x\o)$, and
\[a\phi_\a-\Delta\phi_\a=-\nabla\cdot f\;\;\textrm{in}\;\;\R^d,\]
for $f(x,\omega)=F(\tau_x\o)$.  Due to the $\C^1$-boundedness of $f$, it follows from the Feynman-Kac formula that
\[\phi_\a(x,\o)=\int_0^\infty\int_{\R^d}e^{-\a s}(4\pi s)^{-\nicefrac{d}{2}}\exp(-\nicefrac{\abs{x-y}^2}{4s})\left(-\nabla_y\cdot f(y,\omega)\right) \dy\ds,\]
from which a direct computation proves almost surely that, for $c\in(0,\infty)$ independent of $\a\in(0,1)$,
\begin{equation}\label{cor_24}\norm{\nabla\phi_\a}_{L^\infty(\R^d;\R^d)}\leq \frac{c}{\a}.\end{equation}
Therefore, almost surely,
\begin{align*}
\nabla\phi_\a(x,\omega) & = \int_0^\infty\int_{\R^d} e^{-\a s}(4\pi s)^{-\nicefrac{d}{2}}\nabla_x\nabla_y\left(\exp(-\nicefrac{\abs{x-y}^2}{4s})\right)f(y,\omega) \dy\ds
\\ & = \int_{\R^d}\nabla_x\nabla_y K_\a(x,y)f(y,\omega) \dy,
\end{align*}
for $K_\a(x,y)=\int_0^\infty e^{-\a s}(4\pi s)^{-\nicefrac{d}{2}}\exp\left(-\nicefrac{\abs{x-y}^2}{4s}\right)\ds$.
A direct computation proves that, for $c\in(0,\infty)$ independent of $\a\in(0,1)$,
\[\abs{\nabla_xK_\a(x,y)}+\abs{\nabla_yK_\a(x,y)}\leq c\abs{x-y}^{1-d}e^{-\sqrt{\a}\abs{x-y}},\]
and, for $c\in(0,\infty)$ independent of $\a\in(0,1)$,
\begin{equation}\label{cor_25}
\abs{\nabla_x\nabla_y K_\a(x,y)} \leq c\abs{x-y}^{-d}e^{-\sqrt{\a}\abs{x-y}}.
\end{equation}
Therefore, $\nabla_x\nabla_y K_\a(x,y)$ defines a Calderon-Zygmund kernel (cf.\ eg.\ Stein \cite{Ste1970}).  Let $\eta\colon\R^d\rightarrow [0,1]$ be a smooth function satisfying $\eta = 1$ on $\overline{B}_1$ and $\eta=0$ on $\R^d\setminus B_2$, and for every $R\in(0,\infty)$ let $\eta_R(x)=\eta(\nicefrac{x}{R})$.  For each $R\in(0,\infty)$ let
\[\nabla\phi_{\a,R}(x,\omega)= \int_{\R^d}\nabla_x\nabla_y K_\a(x,y)\eta_R(y)f(y,\omega)\dy.\]
It follows almost surely from \eqref{cor_25} that, for constants $c_1,c_2\in(0,\infty)$ independent of $\a\in(0,1)$,
\begin{equation}\label{cor_26}\sup_{x\in B_{\nicefrac{R}{2}}}\abs{\nabla\phi_\a(x,\omega)-\nabla\phi_{\a,R}(x,\omega)}\leq c_1\norm{F}_{L^\infty(\O)}(\sqrt{\a}R)^{-1}\exp(-c_2\sqrt{\a}R).\end{equation}
It follows from \eqref{cor_25} and the Calderon-Zygmund estimate (cf.\ eg.\ \cite{Ste1970})  that there exists $c\in(0,\infty)$ depending on $p$ and $d$ such that
\begin{equation}\label{cor_27}\int_{\R^d}\abs{\nabla\phi_{\a,R}}^p\leq c\int_{\R^d}\abs{\eta_Rf}^p.\end{equation}
In combination \eqref{cor_26}, \eqref{cor_27}, and the definition of $\eta_R$ prove almost surely that, for every $R\in(0,\infty)$, for $c\in(0,\infty)$ depending on $p$ and $d$ but independent of $R\in(0,\infty)$,
\[\fint_{B_{\frac{R}{2}}}\abs{\nabla\phi_\a}^p\leq c\fint_{B_{2R}}\abs{f}^p+c_1\norm{F}_{L^\infty(\O)}(\sqrt{\a}R)^{-1}\exp(-c_2\sqrt{\a}R).\]
Therefore, after passing to the limit $R\rightarrow\infty$, the ergodic theorem and \eqref{cor_24} prove that, for $c\in(0,\infty)$ depending on $p$ and $d$,
\begin{equation}\label{cor_270}\E\left[\abs{D\Phi_\a}^p\right]\leq c\E\left[\abs{F}^p\right].\end{equation}
Then, after passing to a subsequence $\alpha\rightarrow 0$, the weak lower-semicontinuity of the Sobolev norm proves that there exists $\Phi\in L^2_{\textrm{pot}}(\O)$ satisfying, for $c\in(0,\infty)$ depending on $p$ and $d$,
\[-D\cdot \Phi = -D\cdot F\;\;\textrm{in}\;\;D'(\O)\;\;\textrm{with}\;\;\E\left[\abs{\Phi}^p\right]\leq c\E\left[\abs{F}^p\right].\]
The proof of existence for $F\in L^p(\O;\R^d)$ then follows from the density of $\mathcal{D}(\O)$ in $L^p(\O)$, the definition of $L^2_{\textrm{pot}}(\O)$, and the weak lower-semicontinuity of the Sobolev norm.  It remains to prove uniqueness.

By linearity, it suffices to prove that the only $\Phi\in L^p(\O;\R^d)$ satisfying \eqref{pf_1} and \eqref{pf_2} with $F=0$ is $\Phi=0$.  Since $\Phi$ is mean zero and curl-free let $\phi\colon\R^d\times\O\rightarrow \R$ be the unique function that almost surely satisfies $\fint_{B_1}\phi = 0$, that $\phi\in W^{1,p}_{\textrm{loc}}(\R^d)$ with $\nabla\phi(x,\o)=\Phi(\tau_x\o)$, and that $\phi$ is a weak solution of $-\Delta\phi = 0$ on $\R^d$.  For every $\ve\in(0,1)$ let $\rho^\ve\in\C^\infty_c(\R^d)$ be a standard convolution kernel of scale $\ve$ and let $\phi^\ve=u*\rho^\ve$.  Then $\phi^\ve$ is almost surely harmonic on $\R^d$ and the Feynman-Kac formula proves that there exists $c\in(0,\infty)$ such that, for every $\ve\in(0,1)$ and $t\in(0,\infty)$,
\begin{align}\label{cor_29}
\abs{\nabla \phi^\ve(0)} & =\abs{\int_{\R^d}\phi^\ve(y)(4\pi t)^{-\nicefrac{d}{2}}\frac{y}{2}\exp\left(-\nicefrac{\abs{y}^2}{4t}\right)\dy}
\\ \nonumber & \leq c\left(\int_{\R^d}\frac{\phi^\ve(\sqrt{t}y)}{\sqrt{t}}\abs{y}\exp\left(-\nicefrac{\abs{y}^2}{4}\right)\dy\right).
\end{align}
For each $R\in(0,\infty)$ there exists $c\in(0,\infty)$ independent of $R$ such that
\begin{equation}\label{cor_30}
\abs{\nabla \phi^\ve(0)}\leq c\left(R^d\fint_{B_R} \frac{\phi^\ve(\sqrt{t}y)}{\sqrt{t}}\dy+\int_R^\infty\left(\frac{\phi^\ve(0)}{\sqrt{t}}+\fint_{B_r}\abs{\nabla \phi^\ve(\sqrt{t} y)}\dy\right)r^{2d}e^{-\frac{\abs{r}^2}{4}}\dr\right).
\end{equation}
Proposition~\ref{prop_sublinear}, the ergodic theorem, and $\Phi\in L^p(\O;\R^d)$ prove almost surely for some $c\in(0,\infty)$ that, after passing to the limit $t\rightarrow\infty$,
\[\abs{\nabla \phi^\ve(0)}\leq c\E\left[\abs{\Phi^\ve}\right]\int_R^\infty r^{2d}e^{-\frac{r^2}{4}}\dr,\]
for $\Phi^\ve(\o)=\int_{\R^d}\Phi(\tau_x\o)\rho^\ve(x)\dx$.  After passing to the limit $R\rightarrow \infty$, we conclude almost surely that $\abs{\nabla \phi^\ve(0)}=0$ and therefore by stationarity that $\Phi^\ve=0$.  After passing to the limit $\ve\rightarrow 0$, we conclude that $\Phi=0$.  This completes the proof. \end{proof}

\begin{prop}\label{prop_flux_corrector}  Assume \eqref{steady}.  Let $p_d\in(1,\infty)$ be defined in \eqref{pd_exponent}.  For every $i,j,k\in\{1,\ldots,d\}$ let $\Sigma_{ijk}\in L^{p_d}(\O;\R^d)$ be the unique solution of
\[-D\cdot \Sigma_{ijk}=D_j Q_{ik}-D_k Q_{ij}\;\;\textrm{in}\;\;D'(\O),\]
defined in Proposition~\ref{prop_flux} and let $\sigma_{ijk}\colon\R^d\times\O\rightarrow\R$ be the unique function that almost surely satisfies $\fint_{B_1}\sigma_{ijk}=0$, $\sigma_{ijk}\in W^{1,p_d}_{\textrm{loc}}(\R^d)$ with $\nabla\sigma(x,\o)=\Sigma_{ijk}(\tau_x\o)$, and, for every $\psi\in\C^\infty_c(\R^d)$,
\[\int_{\R^d}\nabla\sigma_{ijk}\cdot\nabla\psi = \int_{\R^d}\partial_k\psi q_{ij}-\partial_j\psi q_{ik},\]
for $q_i(x,\o)=Q_i(\tau_x\o)-\E[Q_i]$.  Then for every $i\in\{1,\ldots,d\}$ the matrix $\sigma_i=(\sigma_{ijk})$ is skew-symmetric and almost surely satisfies
\begin{equation}\label{pfc_1}\nabla\cdot\sigma_i = q_i \;\;\textrm{in}\;\;\R^d\;\;\textrm{for}\;\;(\nabla\cdot \sigma_i)_j=\partial_k\sigma_{ijk}.\end{equation}
\end{prop}

\begin{proof}  Let $i,j,k\in\{1,\ldots,d\}$.  It follows from the uniqueness of Proposition~\ref{prop_flux} that $\Sigma_{ijk}=-\Sigma_{ikj}$ and therefore it follows from the definition of $\sigma_{ijk}$ that $\sigma_{ijk}=-\sigma_{ikj}$.  This proves that $\sigma_i$ is skew-symmetric.  It remains only to prove the equality \eqref{pfc_1}.  This will follow from the distributional equality, for every $j\in\{1,\ldots,d\}$,
\[\Delta\left((\nabla\cdot \sigma_i)_j-q_{ij}\right)=0.\]
Indeed, using the equation satisfied by the $\sigma_{ijk}$ and the fact that $q_i$ is divergence-free, for each $j\in\{1,\ldots,d\}$ we have as distributions that
\begin{equation}\label{pfc_2} \Delta\left((\nabla\cdot \sigma_i)_j-q_{ij}\right)  = \partial_s\partial_s\left(\partial_k\sigma_{ijk}-q_{ij}\right) =\Delta q_{ij}-\partial_k\partial_jq_{ik}-\Delta q_{ij} = -\partial_j(\nabla\cdot q_{ik}) = 0. \end{equation}
Equation \eqref{pfc_2} proves that, for standard convolution kernels $\rho^\ve\in\C^\infty_c(\R^d)$ of scale $\ve\in(0,1)$, for every $j\in\{1,\ldots,d\}$ and $\ve\in(0,1)$,
\[\Delta\left[\left((\nabla\cdot \sigma_i)_j-q_{ij}\right)*\rho^\ve\right]=0\;\;\textrm{in}\;\;\R^d.\]
A repetition of the arguments leading to \eqref{cor_29} and \eqref{cor_30} in the proof of Proposition~\ref{prop_flux} proves that, for each $j\in\{1,\ldots,d\}$ there exists $c^\ve_j\in L^\infty(\O)$ such that almost surely
\[[\left((\nabla\cdot \sigma_i)_j-q_{ij}\right)*\rho^\ve](x,\o)=c_j^\ve(\o)\;\;\textrm{for every}\;\;x\in\R^d.\]
Since the gradient fields $\Sigma_{ijk}$ are mean zero, the stationarity of the gradient, the stationarity of the flux, and the definition of the $q_i$ prove almost surely with the ergodic theorem that, for every $j\in\{1,\ldots,d\}$,
\[0=\lim_{R\rightarrow\infty}\fint_{B_R}\left((\nabla\cdot \sigma_i)_j-q_{ij}\right)*\rho^\ve=c^\ve_j(\o).\]
After passing to the limit $\ve\rightarrow 0$, we have almost surely that
\[\nabla\cdot\sigma_i = q_i\;\;\textrm{in}\;\;\R^d.\qedhere\]
\end{proof}

\subsection{The stream matrix}\label{sec_stream}  In Proposition~\ref{prop_stream} below, we will prove that every mean zero, divergence-free, $L^p$-integrable vector field $B$ satisfying a finite-range of dependence admits an $L^p$-integrable stream matrix provided $p\in[2,\infty)$ and the dimension $d\geq 3$.  We assume a finite range of dependence for simplicity:  that is, there exists $R\in(0,\infty)$ such that for subsets $A_1,A_2\subseteq\R^d$ the sigma algebras
\[\sigma(B(\tau_x\o)\colon x\in A_1)\;\;\textrm{and}\;\;\sigma(B(\tau_x\o)\colon x\in A_2)\;\;\textrm{are independent whenever}\;\;\dd(A_1,A_2)\geq R.\]
In the case $p=2$, for instance, the same proof yields the existence of a stationary stream matrix provided the spatial correlations of $B$ decay faster than a square.

\begin{prop}\label{prop_stream} Assume~\eqref{steady}.  Let $d\in[3,4,\ldots)$, let $p\in[2,\infty)$, and let $B\in L^p(\O;\R^d)$ satisfy $\E[B]=0$, $D\cdot B=0$, and, for some $R\in(0,\infty)$, for every $A_1,A_2\subseteq\R^d$,
\begin{equation}\label{ps_1}\sigma(B(\tau_x\o)\colon x\in A_1)\;\;\textrm{and}\;\;\sigma(B(\tau_x\o)\colon x\in A_2)\;\;\textrm{are independent whenever}\;\;\dd(A_1,A_2)\geq R.\end{equation}
Then there exists skew-symmetric matrix $S=(S_{jk})_{j,k\in\{1,\ldots,d\}}\in L^p(\O;\R^{d\times d})$ that satisfies
\[D\cdot S = B\;\;\textrm{in}\;\;L^p(\O;\R^d).\]
\end{prop}

\begin{proof}  Let $\F_B$ denote the sigma algebra generated by $B$.  It follows from \eqref{ps_1} that every $\F_B$-measurable random variable satisfies a finite range of dependence.  Let $X=(X_i)_{i\in\{1,\ldots,d\}}\in L^\infty(\O;\R^d)$ be $\F_B$-measurable and for every $\a\in(0,1)$ let $S_\a\in\mathcal{H}^1(\O)$ denote the unique Lax-Milgram solution of the equation
\[\a S_\a-D\cdot DS_\a=D\cdot X.\]
Due to the boundedness of $X$, we have the representation
\begin{align}\label{ps_2}
S_\a(\o) & =\int_0^\infty \int_{\R^d}(4\pi s)^{-\frac{d}{2}}e^{-\a s-\frac{\abs{x}^2}{4s}}\frac{x}{2s}\cdot X(\tau_x\o)\dx\ds
\\ \nonumber  & =(4\pi)^{-\frac{d}{2}}\int_0^\infty \int_{\R^d}\abs{x}^{1-d}s^{-(\frac{d}{2}+1)}e^{-\a s\abs{x}^2-\frac{1}{4s}}\frac{x}{2\abs{x}}\cdot X(\tau_x\o)\dx\ds.
\end{align}
Let $q\in\{2,4,6,\ldots\}$ be a nonzero even integer and let $\mathcal{I}_q$ denote the collection of partitions of $\{1,2,\ldots,q\}$ of the form
\[\mathcal{I}_q=\{\beta=\left(\beta_1,\ldots,\beta_{N(\beta)}\right)\colon \beta_j\in\{2,3,\ldots\}\;\forall\;j\in\{1,\ldots,N(\beta)\},\;\textrm{and}\;\sum_{j=1}^{N(\beta)}\beta_j=q\},\]
which are exactly the partitions of $\{1,2,\ldots,q\}$ that contain no singletons.  We define for every $\beta\in\mathcal{I}_q$ the integral
\[I_\beta = \prod_{j=1}^{N(\beta)}\int_{\R^d}\int_{B_{Rq}(x_{\beta_j})^{\beta_j-1}}\prod_{k=\beta_{j-1}+1}^{\beta_j}\abs{x_k}^{1-d}\dx_{\beta_{j-1}+1}\ldots\dx_{\beta_j}\]
and observe from the assumption $d\geq 3$ and the fact that $\beta_j\geq 2$ for every $j\in\{1,\ldots,N(\beta)\}$ that, for some $c\in(0,\infty)$ independent of $\a\in(0,1)$, for every $\beta\in\mathcal{I}_q$,
\[I_\beta \leq c \prod_{j=1}^{N(\beta)}\int_{\R^d}\left(1\wedge \abs{x_{\beta_j}}^{1-d}\right)^{\beta_j}\dx_{\beta_j}\leq c\left(\int_0^\infty (1\wedge r^{(1-d)})\dr\right)^{N(\beta)}<\infty.\]
It then follows from the $\F_B$-measurability of $X$, \eqref{ps_1}, the fact that transformation group preserves the measure, H\"older's inequality, and an explicit calculation based on \eqref{ps_2} that, for some $c\in(0,\infty)$ independent of $\a\in(0,1)$,
\begin{equation}\label{ps_3}\E\left[S_\a^q\right]\leq c\E\left[\abs{X}^q\right]\sum_{\beta\in\mathcal{I}_q}I_\beta\leq c\E\left[\abs{X}^q\right].\end{equation}
Since it follows from \eqref{cor_270} that, for some $c\in(0,\infty)$ independent of $\a\in(0,1)$,
\begin{equation}\label{ps_4} \E\left[\abs{DS_\a}^q\right]\leq c\E\left[\abs{X}^q\right],\end{equation}
it follows after passing to a subsequence $\a\rightarrow 0$ that there exists $S\in L^q(\O)\cap \mathcal{H}^1(\O)$ with $DS\in L^q(\O;\R^d)$ such that
\[S_\a\rightharpoonup S\;\;\textrm{weakly in}\;\;L^q(\O)\;\;\textrm{and}\;\;DS_\a\rightharpoonup DS\;\;\textrm{weakly in}\;\;L^q(\O;\R^d).\]
It follows from Proposition~\ref{prop_flux}, \eqref{ps_3}, \eqref{ps_4}, and the weak lower-semicontinuity of the Sobolev norm that $DS\in L^q(\O;\R^d)$ is the unique curl-free, mean zero solution of
\begin{equation}\label{ps_6}-D\cdot DS = D\cdot X\;\;\textrm{in}\;\;D'(\O),\end{equation}
and that, for some $c\in(0,\infty)$ depending on $q\in\{2,4,6,\ldots\}$ but independent of $X$,
\begin{equation}\label{ps_7}\E\left[\abs{S}^q+\abs{DS}^q\right]\leq c\E\left[\abs{X}^q\right].\end{equation}
The density of bounded functions in $L^q(\O)$ for every $q\in\{2,4,6,\ldots\}$ proves that, for every $\F_B$-measurable $X\in L^q(\O;\R^d)$ there exists a unique $S\in L^q(\O)\cap \mathcal{H}^1(\O)$ with $DS\in L^q(\O;\R^d)$ that satisfies \eqref{ps_6} and \eqref{ps_7}.  Finally, since $q\in\{2,4,6,\ldots\}$ was arbitrary, it follows from the Riesz-Thorin interpolation theorem applied to the spaces $L^p(\O,\F_B)$ for $p\in[2,\infty)$ that for every $\F_B$-measurable $X\in L^p(\O;\R^d)$ there exists a unique $S\in L^p(\O)\cap\mathcal{H}^1(\O)$ with $DS\in L^p(\O;\R^d)$ satisfying \eqref{ps_6} and \eqref{ps_7}.

Now let $B=(B_i)_{i\in\{1,\ldots,d\}}\in L^p(\O;\R^d)$ for some $p\in[2,\infty)$ be mean zero and divergence-free in the sense that $\E[B]=0$ and $D\cdot B=0$, and let $B$ satisfy a finite range of dependence.  For every $j,k\in\{1,\ldots,d\}$ let $S_{jk}\in L^p(\O)\cap\mathcal{H}^1(\O)$ be the unique solution of
\[-D\cdot DS_{jk}=D_jB_k-D_kB_j.\]
The uniqueness proves that $S_{jk}=-S_{kj}$ for every $j,k\in\{1,\ldots,d\}$ and it follows from Proposition~\ref{prop_flux_corrector} and $\E[B]=0$ that for $S=(S_{jk})_{j,k\in\{1,\ldots,d\}}$ we have $D\cdot S=B$ in $L^p(\O;\R^d)$.  This completes the proof.  \end{proof}

\section{Quenched stochastic homogenization}\label{sec_stoch_hom}

In this section, we will prove the quenched stochastic homogenization of the equation
\begin{equation}\label{hom_2}-\nabla\cdot(a^\ve+s^\ve)\nabla u^\ve=f\;\;\textrm{in}\;\;U\;\;\textrm{with}\;\;u^\ve=g\;\;\textrm{on}\;\;\partial U.\end{equation}
The proof is based on estimating the energy of the two-scale expansion
\[w^\ve=u^\ve-v-\ve\phi^\ve_i\partial_i v,\]
where, for the gradient fields $\Phi_i\in L^2_{\textrm{pot}}(\O)$ constructed in Proposition~\ref{prop_corrector}, the physical correctors $\phi_i$ are the unique functions that almost surely satisfy $\fint_{B_1}\phi_i=0$, $\phi_i\in H^1_{\textrm{loc}}(\R^d)$ with $\nabla\phi_i(x,\o)=\Phi_i(\tau_x\o)$, and, for every $\psi\in\C^\infty_c(\R^d)$,
\begin{equation}\label{hom_03}\int_{\R^d}(a+s)(\nabla\phi_i+e_i)\cdot\nabla\psi=0,\end{equation}
and where $\phi_i^\ve(x,\o)=\phi(\nicefrac{x}{\ve},\o)$.  The limit $v\in H^1(U)$ solves the homogenized equation
\begin{equation}\label{hom_3}-\nabla\cdot\overline{a}\nabla v = f\;\;\textrm{in}\;\;U\;\;\textrm{with}\;\; u=f\;\;\textrm{on}\;\;\partial U,\end{equation}
for the homogenized coefficient field $\overline{a}\in\R^{d\times d}$ defined for each $i\in\{1,\ldots,d\}$ by
\begin{equation}\label{hom_0003}\overline{a}e_i=\E\left[(A+S)(\Phi_i+e_i)\right].\end{equation}
Motivated by the analogous computation in \cite{GloNeuOtt2020}, after introducing the flux correctors $\sigma_i$ we will prove that, up to boundary terms,
\[-\nabla\cdot (a^\ve+s^\ve)\nabla w^\ve=\nabla\cdot\left[(\ve\phi^\ve_i(a^\ve+s^\ve)-\ve\sigma^\ve_i)\nabla(\partial_i v)\right].\]
The strong convergence of $\nabla w^\ve$ to zero in the $\ve\rightarrow 0$ limit then follows formally from the $L^{d\vee(2+\d)}$-integrability of the stream matrix, Proposition~\ref{prop_sublinear}, and the regularity of $\overline{a}$-harmonic functions.

This section is organized as follows.  We prove the well-posedness of \eqref{hom_2} in Proposition~\ref{prop_hom_wp} below.  We prove that $\overline{a}$ is uniform elliptic in Proposition~\ref{prop_hom_ue} below, a fact which relies on the energy identity \eqref{cor_008}.  Finally, we prove the quenched homogenization of \eqref{hom_2} in Theorem~\ref{thm_ts} below.

\begin{prop}\label{prop_hom_wp}  Let $U\subseteq\R^d$ be a bounded $\C^{2,\a}$-domain for some $\a\in(0,1)$, let $a\in L^\infty(U;\R^{d\times d})$ be uniformly elliptic, and let $s=(s_{jk})\in H^1(U;\R^{d\times d})$ be skew-symmetric.  Then for every $f_1\in L^2(U)$, $f_2\in L^2(U;\R^d)$, and $g\in W^{1,\infty}(\partial U)$ there exists a unique weak solution $u\in H^1(U)$ of the equation
\begin{equation}\label{hom_5}-\nabla\cdot (a+s)\nabla u = f_1+\nabla\cdot f_2\;\;\textrm{in}\;\;U\;\;\textrm{with}\;\; u = g\;\;\textrm{on}\;\;\partial U.\end{equation}
\end{prop}

\begin{proof}  The regularity of the domain $U$ and the tubular neighborhood theorem prove that there exists a globally Lipschitz continuous function $\overline{g}\colon\R^d\rightarrow \R$ such that $\overline{g}|_{\partial U}=g$.  Then by considering $\tilde{u}=u-\overline{g}$ it follows that $u\in H^1(U)$ solves \eqref{hom_5} if and only if $\tilde{u}\in H^1_0(U)$ solves
\[-\nabla\cdot (a+s)\nabla \tilde{u} = f_1+\nabla\cdot \tilde{f}_2\;\;\textrm{in}\;\;U\;\;\textrm{with}\;\; u = 0\;\;\textrm{on}\;\;\partial U,\]
for $\tilde{f}_2=f_2+a\nabla\overline{g}+s\nabla\overline{g}\in L^2(U;\R^d)$.  It is therefore sufficient to consider the case $g=0$.

Let $f_1\in L^2(U)$ and $f_2\in L^2(U;\R^d)$ and for each $n\in\N$ let $s_n=(s^n_{jk})$ be defined by $s^n_{jk}=(s_{jk}\wedge n)\vee(-n)$.  The Lax-Milgram theorem proves for every $n\in\N$ that there exists a unique solution $u_n\in H^1_0(U)$ of the equation
\begin{equation}\label{hom_7}-\nabla\cdot (a+s_n)\nabla u_n = f_1+\nabla\cdot f_2\;\;\textrm{in}\;\;U\;\;\textrm{with}\;\; u = 0\;\;\textrm{on}\;\;\partial U,\end{equation}
which due to the skew-symmetry of $s_n$, the uniform ellipticity, and the Poincar\'e inequality satisfies the energy inequality, for some $c\in(0,\infty)$ independent of $n$,
\begin{equation}\label{hom_8}\int_U\abs{\nabla u_n}^2\leq c\int_U\abs{f_1}^2+\abs{f_2}^2.\end{equation}
Therefore, after passing to a subsequence $n\rightarrow\infty$, it follows from \eqref{hom_7}, \eqref{hom_8}, and the strong convergence of $s_n$ to $s$ in $L^2(U;\R^{d\times d})$ that there exists $u\in H^1_0(U)$ such that $u_n\rightharpoonup u$ weakly in $H^1_0(U)$ and such that, for every $\psi\in\C^\infty_c(U)$,
\[\int_U(a+s)\nabla u\cdot\nabla\psi = \int_U f_1\psi -f_2\cdot \nabla \psi.\]
It remains to prove the uniqueness of $u$.  By linearity, it suffices to prove that the only $u\in H^1_0(U)$ that solves \eqref{hom_9} with $f_1=0$ and $f_2=0$ is $u=0$.  Since $s\in H^1(U;\R^{d\times d})$ is skew symmetric and since $\nabla\cdot s$ is divergence-free, we have, for every $\psi\in\C^\infty_c(U)$,
\[\int_Us\nabla u\nabla\psi = -\int_U (\nabla\cdot s)\cdot\nabla\psi u = \int_U(\nabla\cdot s)\cdot\nabla u\psi.\]
Therefore, for every $\psi\in\C^\infty_c(U)$,
\begin{equation}\label{hom_9}\int_U a\nabla u\cdot\nabla\psi + \int_U(\nabla\cdot s)\nabla u\psi = 0.\end{equation}
It follows as in the proof of Proposition~\ref{prop_corrector} that for each $n\in\N$ the function $u_n=(u\wedge n)\vee(-n)$ is an admissible test function for \eqref{hom_9}.  The distributional equalities $\nabla u_n=\nabla u\mathbf{1}_{\{\abs{u}\leq n\}}$ and $\nabla u u_n = \nabla(uu_n-\nicefrac{1}{2}u_n^2)$ and the fact that $\nabla\cdot s$ is divergence-free then prove, for each $n\in\N$,
\[\int_Ua\nabla u\cdot\nabla u\mathbf{1}_{\{\abs{u}\leq n\}}=0.\]
After passing to the limit $n\rightarrow\infty$, we conclude using the uniform ellipticity and the monotone convergence theorem that $\nabla u=0$ and therefore that $u=0$.  This completes the proof.  \end{proof}

\begin{prop}\label{prop_hom_ue}  Assume \eqref{steady}.  Let $\overline{a}\in\R^{d\times d}$ be defined by \eqref{hom_0003}.   Then, for every $\xi\in\R^d$,
\[\abs{\overline{a}\xi}\leq 2 \left(\Lambda+\E\left[\abs{S}^2\right]^\frac{1}{2}\right)\left(\sum_{i=1}^d\E[\abs{\Phi_i+e_i}^2]\right)^\frac{1}{2}\abs{\xi}\;\;\textrm{and}\;\;\overline{a}\xi\cdot\xi\geq\lambda\abs{\xi}^2.\]
\end{prop}

\begin{proof}  The uniform ellipticity, H\"older's inequality, the linearity, and the definition of $\overline{a}$ prove that, for every $\xi\in\R^d$,
\begin{align*}
\abs{\overline{a}\xi} =\abs{\xi_i\E\left[(A+S)(\Phi_i+e_i)\right]} & \leq 2\left(\Lambda+\E\left[\abs{S}^2\right]^\frac{1}{2}\right)\abs{\xi_i}\E\left[\abs{\Phi_i+e_i}^2\right]^\frac{1}{2}
\\ & \leq 2 \left(\Lambda+\E\left[\abs{S}^2\right]^\frac{1}{2}\right)\left(\sum_{i=1}^d\E[\abs{\Phi_i+e_i}^2]\right)^\frac{1}{2}\abs{\xi}.
\end{align*}
Similarly it follows by definition that
\[\overline{a}\xi\cdot \xi = \xi_i^2\E\left[(A+S)(\Phi_i+e_i)\cdot e_i\right],\]
and it follow from the skew-symmetry of $S$, the energy identity \eqref{cor_008}, the fact that $-D\cdot (A+S)\Phi_i=0$ in $L^2_{\textrm{pot}}(\O)$, the uniform ellipticity, Jensen's inequality, and $\E[\Phi_i]=0$ that
\begin{align*}
\overline{a}\xi\cdot \xi & = \xi_i^2\E\left[(A+S)(\Phi_i+e_i)\cdot (\Phi_i+e_i)\right] =\xi_i^2\E\left[A(\Phi_i+e_i)\cdot (\Phi_i+e_i)\right],
\\ & \geq \lambda \xi_i^2\E\left[\abs{\Phi_i+e_i}^2\right] \geq \lambda\xi^2_i\abs{\E[\Phi_i+e_i]}^2=\lambda\abs{\xi}^2, &
\end{align*}
which completes the proof. \end{proof}

\begin{thm}\label{thm_ts}  Assume \eqref{steady}.  Let $\a\in(0,1)$, let $U\subseteq\R^d$ be a bounded $\C^{2,\a}$-domain, let $f\in\C^\a(U)$, and let $g\in\C^{2,\a}(\partial U)$.  For every $\ve\in(0,1)$ let $u^\ve\in H^1(U)$ be the unique solution of \eqref{hom_2} and let $v\in H^1(U)$ be the unique solution of \eqref{hom_3}.  Then, almost surely as $\ve\rightarrow 0$,
\[\lim_{\ve\rightarrow 0}\norm{u^\ve-v-\ve\phi^\ve_i\partial_iv}_{H^1(U)}=0.\]
\end{thm}

\begin{proof}  We will essentially study the equation satisfied by the homogenization error
\begin{equation}\label{sh_29}w^\ve=u^\ve-v-\ve\phi^\ve_i\partial_iv,\end{equation}
after introducing a cutoff to ensure that $w^\ve$ vanishes along the boundary.  Using the fact that $U\subseteq\R^d$ is a bounded $\C^{2,\a}$-domain, for each $\rho\in(0,1)$ we define $\eta_\rho\colon\overline{U}\rightarrow[0,1]$ to be a smooth cutoff function satisfying $\eta_\rho(x)=1$ if $\dd(x,\partial U)\geq 2\rho$, $\eta_\rho(x)=0$ if $\dd(x,\partial U)<\rho$, and $\abs{\nabla\eta_\rho(x)}\leq \nicefrac{c}{\rho}$ for some $c\in(0,\infty)$ independent of $\rho\in(0,1)$.  For each $\ve,\rho\in(0,1)$ we define
\[w^{\ve,\rho}=u^\ve-v-\ve\phi^\ve_i\eta_\rho\partial_iv\;\;\textrm{in}\;\;H^1_0(U).\]
It follows by definition that
\[\nabla w^{\ve,\rho}=\nabla u^\ve-\nabla v-\eta_\rho\partial_iv\nabla\phi^\ve_i-\ve\phi^\ve_i\nabla(\eta_\rho\partial_i v).\]
Distributionally, using the equation satisfied by $u^\ve$,
\[-\nabla\cdot (a^\ve+s^\ve)\nabla w^{\ve,\rho} = f+\nabla\cdot (a^\ve+s^\ve)\nabla v+\nabla\cdot(a^\ve+s^\ve)\left(\eta_\rho\partial_iv\nabla\phi^\ve_i +\ve\phi^\ve\nabla(\eta_\rho\partial_iv)\right),\]
and, using the equation satisfied by $v$,
\begin{align}\label{sh_30} -\nabla\cdot (a^\ve+s^\ve)\nabla w^{\ve,\rho} & = \nabla\cdot\left[ (1-\eta_\rho)\left((a^\ve+s^\ve)-\overline{a}\right)\nabla v\right]
\\ \nonumber & \quad +\nabla \cdot\left[\left((a^\ve+s^\ve)(\nabla\phi_i+e_i)-\overline{a}e_i\right)\eta_\rho\partial_iv\right]
\\ \nonumber & \quad + \nabla\cdot\left[(a^\ve+s^\ve)\ve\phi^\ve_i\nabla(\eta_\rho\partial_iv)\right].
\end{align}
The second term on the righthand side of \eqref{sh_30} is defined for each $i\in\{1,\ldots,d\}$ by $q^\ve_i(x,\o)=Q_i(\tau_{\nicefrac{x}{\ve}}\o)-\E\left[Q_i\right]$ for the flux $Q_i$ defined by
\[Q_i=(A+S)(\Phi_i+e_i)\;\;\textrm{in}\;\;L^{p_d}(\O;\R^d),\]
for $p_d$ defined in \eqref{pd_exponent}.  The $q^\ve_i$ do not vanish in a strong sense as $\ve\rightarrow 0$, and it is for this reason that we introduce the flux correctors defined in Proposition~\ref{prop_flux_corrector}.  For each $i\in\{1,\ldots,d\}$ let $\sigma_i=(\sigma_{ijk})\in W^{1,p_d}_{\textrm{loc}}(\R^d;\R^{d\times d})$ be as in Proposition~\ref{prop_flux_corrector} and let $\sigma_i^\ve(x,\o)=\sigma_i(\nicefrac{x}{\ve},\o)$.  Then, for every $\psi\in \C^\infty_c(U)$,
\begin{align*}
\int_{\R^d}q^\ve_i\eta_\rho\partial_i v\cdot\nabla\psi & = \int_{\R^d}(\eta_\rho\partial_i v)q^\ve_{ij}\partial_j\psi=\int_{\R^d}(\eta_\rho\partial_iv)\partial_k(\ve\sigma^\ve_{ijk})\partial_j\psi
\\ & =-\int_{\R^d}\ve\sigma^\ve_{ijk}\partial_k(\eta_\rho\partial_i v)\partial_j\psi,
\end{align*}
where the final inequality relies on the skew-symmetry.  So, as distributions on $\R^d$,
\[\nabla\cdot\left[q^\ve_i\eta_\rho\partial_i v\right]=-\nabla\cdot\left[\ve\sigma^\ve_i\nabla(\eta_\rho\partial_i v)\right].\]
Returning to \eqref{sh_30}, we conclude that
\begin{align}\label{sh_32}
& -\nabla\cdot (a^\ve+s^\ve)\nabla w^{\ve,\rho}
\\ & \nonumber = \nabla\cdot\left[ (1-\eta_\rho)\left((a^\ve+s^\ve)-\overline{a}\right)\nabla v\right] + \nabla\cdot\left[\left(\ve\phi^\ve_i(a^\ve+s^\ve)-\ve\sigma^\ve_i\right)\nabla(\eta_\rho\partial_iv)\right].
\end{align}
The uniform ellipticity, H\"older's inequality, Young's inequality, and the definition of $\eta_\rho$ prove that, for some $c\in(0,\infty)$ independent of $\ve,\rho\in(0,1)$, for $q_d=d\vee(2+\d)$ and $\nicefrac{1}{2_*}=\nicefrac{1}{2}-\nicefrac{1}{q_d}$,
\begin{align*}\int_U\abs{\nabla w^{\ve,\rho}}^2 & \leq c\norm{\nabla v}_{L^\infty(U;\R^d)}^2\left(\int_U(1-\eta_\rho)^2\left(\abs{a^\ve}^2+\abs{s^\ve}^2\right)\right)
\\ \nonumber & \quad +c\norm{\nabla (\eta_\rho\partial_iv)}_{L^\infty(U;\R^d)}^2\left(\int_U \abs{a^\ve}^{q_d}+\abs{s^\ve}^{q_d}\right)^\frac{2}{q_d}\left(\int_U\abs{\ve\phi^\ve_i}^{2_*}\right)^\frac{2}{2_*}
\\ \nonumber & \quad +c\norm{\nabla (\eta_\rho\partial_iv)}^2_{L^\infty(U;\R^d)}\left(\int_U\abs{\ve\sigma^\ve_i}^2\right).
\end{align*}
The regularity of the domain and Schauder estimates (cf.\ eg.\ Gilbarg and Trudinger \cite[Chapter~6]{GilTru2001}) prove that, for some $c\in(0,\infty)$ depending on $U$,
\begin{equation}\label{sh_34}\norm{v}_{\C^{2,\alpha}(U)}\leq c\left(\norm{f}_{\C^\a(U)}+\norm{g}_{\C^{2,a}(\partial U)}\right).\end{equation}
It follows almost surely from Proposition~\ref{prop_sublinear}, \eqref{sh_34}, $\Phi_i\in L^2_{\textrm{pot}}(\O)$, $\Sigma_{ijk}\in L^{p_d}(\O;\R^d)$, $S\in L^{q_d}(\O;\R^{d\times d})$, the uniform ellipticity, the ergodic theorem, and the definition of $\eta_\rho$ that, for each $\rho\in(0,1)$, for $c\in(0,\infty)$ depending on $U$ but independent of $\rho\in(0,1)$,
\begin{equation}\label{sh_35}\limsup_{\ve\rightarrow 0}\int_U\abs{\nabla w^{\ve,\rho}}^2\leq c\rho\norm{\nabla v}^2_{L^\infty(U;\R^d)}\E\left[\abs{A}^2+\abs{S}^2\right].\end{equation}
Then for each $\ve\in(0,1)$ let $w^\ve\in H^1(U)$ be defined by \eqref{sh_29} and for every $\rho\in(0,1)$ observe that
\[\nabla w^\ve=\nabla w^{\ve,\rho}+\nabla \phi^\ve_i(1-\eta_\rho)\partial_iv+\ve\phi^\ve_i\nabla\left((1-\eta_\rho)\partial_i v)\right).\]
It follows from \eqref{sh_35}, the triangle inequality, and Young's inequality that, for $c\in(0,\infty)$ independent of $\ve,\rho\in(0,1)$,
\begin{align*}
& \int_U\abs{\nabla w^\ve}^2
\\ & \leq c\left(\int_U\abs{\nabla w^{\ve,\rho}}^2+\norm{\partial_i v}^2_{L^\infty}\int_U(1-\eta_\rho)^2\abs{\nabla\phi^\ve_i}^2+\norm{\nabla((1-\eta_\rho)\partial_iv)}^2_{L^\infty}\int_U\abs{\ve\phi^\ve_i}^2 \right).
\end{align*}
Proposition~\ref{prop_sublinear}, \eqref{sh_34}, the definition of $\eta_\rho$, and the ergodic theorem therefore prove almost surely that for every $\rho\in(0,1)$, for $c\in(0,\infty)$ depending on $U$ but independent of $\rho$,
\[\limsup_{\ve\rightarrow 0} \int_U\abs{\nabla w^\ve}^2\leq c\rho\left(\norm{\nabla v}^2_{L^\infty(U;\R^d)}E\left[\abs{A}^2+\abs{S}^2\right]+\norm{\partial_i v}^2_{L^\infty(U)}\E\left[\abs{\Phi_i}^2\right]\right).\]
Passing to the limit $\rho\rightarrow 0$, we conclude that, almost surely as $\ve\rightarrow 0$,
\begin{equation}\label{sh_36}\nabla w^\ve\rightarrow 0\;\;\textrm{strongly in}\;\;L^2(U;\R^d).\end{equation}
Finally, since Proposition~\ref{prop_sublinear} and $\Phi_i\in L^2_{\textrm{pot}}(\O)$ prove that, almost surely as $\ve\rightarrow 0$,
\[\ve\phi^\ve_i\partial_i v\rightarrow 0\;\;\textrm{strongly in}\;\;L^2(U),\]
it follows from the uniform boundedness of the $u^\ve-v$ in $H^1_0(U)$, the Sobolev embedding theorem, and \eqref{sh_36} that, almost surely as $\ve\rightarrow 0$,
\begin{equation}\label{sh_38} w^\ve\rightarrow 0\;\;\textrm{strongly in}\;\;L^2(U).\end{equation}
In combination \eqref{sh_36} and \eqref{sh_38} complete the proof.  \end{proof}

\section{The large-scale regularity estimate}\label{sec_lsr}

In this section, motivated by the methods of \cite{GloNeuOtt2020}, we will establish an almost sure intrinsic large-scale $\C^{1,\a}$-regularity estimate for solutions $u\in H^1_{\textrm{loc}}(\R^d)$ of
\begin{equation}\label{ls_1}-\nabla\cdot(a+s)\nabla u=0\;\;\textrm{in}\;\;\R^d.\end{equation}
In analogy with the the characterization of H\"older spaces by Morrey and Campanato, for each $\a\in(0,1)$ and $R\in(0,\infty)$ we define the excess $\Exc(u;R)$ to be the large-scale $\C^{1,\a}$-Campanato semi-norm with respect to the intrinsic $(a+s)$-harmonic coordinates $(x_i+\phi_i)$:
\[\Exc(u;R) = \inf_{\xi\in\R^d}\frac{1}{R^{2\a}}\fint_{B_R}\abs{\nabla u-\xi-\nabla\phi_\xi}^2.\]
Formally the homogenization of \eqref{ls_1} in $H^1(U)$ and the ergodic theorem imply that the excess is well-controlled for large radii $R$ by the regularity of an $\overline{a}$-harmonic function and the energy of the random gradient fields $\Phi_i$.  The arguments of this section make this precise.

The section is organized as follows.  In Propositions~\ref{prop_ue_int} and \ref{prop_ue_est} below we recall some standard results from constant-coefficient elliptic regularity theory.  We estimate the energy of the two-scale expansion in Proposition~\ref{prop_hom_energy} below.  We then prove the large-scale H\"older estimate and excess decay in Proposition~\ref{prop_excess} and Theorem~\ref{thm_excess} below.  The proof of excess decay is most closely related to the methods of \cite{GloNeuOtt2020} in the uniformly elliptic setting, and shares aspects of the work \cite{BelFehOtt2018} in the degenerate elliptic setting.  Here, in analogy with the degenerate setting, the regularity estimate comes into effect after controlling both the sublinearity of the correctors and the large-scale averages of the unbounded stream matrix.  In this way Propositions~\ref{prop_hom_energy} and \ref{prop_excess} are wholly analytic and essentially deterministic, taking as input only this large-scale behavior.  Theorem~\ref{thm_excess} combines these statements with the probabilistic input of Proposition~\ref{prop_sublinear} and the ergodic theorem to obtain the complete statement.

\begin{remark}  In this section, we will write $a\lesssim b$ if $a\leq cb$ for a constant $c$ depending only on the dimension and  ellipticity constants.\end{remark}

\begin{prop}\label{prop_ue_int} Let $\overline{a}\in\R^{d\times d}$ be uniformly elliptic and let $v\in H^1_{\textrm{loc}}(\R^d)$ be a weak solution of the equation
\begin{equation}\label{ue_1}-\nabla\cdot \overline{a}\nabla v = 0\;\;\textrm{in}\;\;B_1.\end{equation}
Then, for each $r_1< r_2\in (0,1)$ and $c\in\R$,
\[\int_{B_{r_1}}\abs{\nabla v}^2\lesssim \frac{1}{(r_2-r_1)^2}\int_{B_{r_2}}(v-c)^2.\]
And, for every $\rho\in(0,1)$,
\[\sup_{B_{(1-\rho)}}\left(\abs{\nabla^2 v}^2+\frac{1}{\rho^2}\abs{\nabla v}^2\right)\lesssim \frac{1}{\rho^{2(d+1)}}\int_{B_1}\abs{\nabla v}^2.\]
\end{prop}

\begin{proof}  Let $r_1<r_2\in(0,1)$ and let $\eta\colon \R^d \rightarrow[0,1]$ be a smooth cutoff function satisfying $\eta=1$ on $\overline{B}_{r_1}$ and $\eta=0$ on $\R^d\setminus B_{r_2}$ with $\abs{\nabla\eta}\leq \nicefrac{2}{(r_2-r_1)}$.    After testing \eqref{ue_1} with $\eta^2(v-c)$,
\[\int_{B_1}\left(\overline{a}\nabla v\cdot\nabla v\right)\eta^2 = -2\int_{B_1}\left(\overline{a}\nabla v\cdot \nabla\eta_r\right)(v-c)\eta.\]
The uniform ellipticity, H\"older's inequality, and Young's inequality prove using the definition of $\eta$ that
\begin{equation}\label{ue_2}\int_{B_{r_1}}\abs{\nabla v}^2\lesssim \frac{1}{(r_2-r_1)^2}\int_{B_{r_2}}(v-c)^2.\end{equation}
Let $K\in\N$ and $\rho\in(0,1)$.  Since for every multi-index $\alpha=(\alpha_1,\ldots,\alpha_d)\in\N_0^d$ the partial derivative $\partial^{\a_1}_1\ldots\partial^{\a_d}_d v$ satisfies \eqref{ue_1}, a repeated application of \eqref{ue_2} on the subintervals of length $\nicefrac{\rho}{2K}$ proves that, for every $k\in\{1,\ldots,K\}$,
\[\int_{B_{(1-\rho)}}\abs{\nabla^kv}^2\lesssim \frac{4K^2}{\rho^2}\int_{B_{\left(1-\rho+\frac{\rho}{2K}\right)}}\abs{\nabla^{k-1}v}^2\lesssim\ldots\lesssim \frac{(4K)^{k-1}}{\rho^{2(k-1)}}\int_{B_{\left(1-\rho+\frac{(k-1)\rho}{2K}\right)}}\abs{\nabla v}^2.\]
After choosing $K=d+2$ the Sobolev embedding theorem proves that
\[\sup_{B_{(1-\rho)}}\left(\abs{\nabla^2 v}^2+\frac{1}{\rho^2}\abs{\nabla v}^2\right)\lesssim \frac{1}{\rho^{2(d+1)}}\int_{B_1}\abs{\nabla v}^2.\qedhere\] \end{proof}

\begin{prop}\label{prop_ue_est}  Let $\overline{a}\in\R^{d\times d}$ be uniformly elliptic, let $\psi\in\C^\infty(B_1)$, and let $v\in H^1(B_1)$ be a weak solution of the equation
\[-\nabla\cdot \overline{a}\nabla v = 0\;\;\textrm{in}\;\;B_1\;\;\textrm{with}\;\;v=\psi\;\;\textrm{on}\;\;\partial B_1.\]
Then, for every $p\in[2,\infty)$ there exists $c_1,c_2\in(0,\infty)$ depending on $p$ such that
\[\norm{\nabla v}_{L^p(B_1)}\leq c_1  \norm{\nabla^{\textrm{tan}}\psi}_{L^p(\partial B_1)}\leq c_2 \norm{\nabla^{\textrm{tan}}\psi}_{L^\infty(B_1)},\]
where $\nabla^{\textrm{tan}}$ denotes the tangential derivative on $\partial B_1$.
\end{prop}

\begin{proof}  We may assume without loss of generality that $\fint_{\partial B_1}\psi = 0$ since subtracting a constant does not change the gradient.  Let $\eta\colon\R \rightarrow[0,1]$ be a smooth function satisfying $\eta=1$ on $[\nicefrac{3}{4},\infty)$ and $\eta=0$ on $(-\infty,\nicefrac{1}{4}]$.  Then $\overline{\psi}(x)=\psi(\nicefrac{x}{\abs{x}})\eta(\abs{x})$ is a smooth extension of $\psi$ into $B_1$ that satisfies, using the fact that $\psi$ has average zero on $\partial B_R$,
\[\norm{\nabla\overline{\psi}}_{L^p(B_1)}\lesssim \norm{\nabla^{\textrm{tan}}\psi}_{L^p(\partial B_1)}\lesssim \norm{\nabla^{\textrm{tan}}\psi}_{L^\infty(\partial B_1)}.\]
It then follows from \cite[Theorem~7.1]{GiaMar2012} and $p\geq 2$ that, for some $c_1,c_2\in(0,\infty)$ depending on $p$,
\[\norm{\nabla v}_{L^p(B_1)}\leq c_1 \norm{\nabla^{\textrm{tan}}\psi}_{L^p(\partial B_1)}\leq c_2\norm{\nabla^{\textrm{tan}}\psi}_{L^\infty(\partial B_1)}.\qedhere\]\end{proof}

\begin{prop}\label{prop_hom_energy}  Assume \eqref{steady}.  Let $R\in(0,\infty)$ and let $u\in H^1(B_R)$ be a distributional solution of
\[-\nabla\cdot (a+s)\nabla u=0\;\;\textrm{in}\;\;B_R.\]
Then there exists $c\in(0,\infty)$ so that for every $\ve\in(0,1)$ there exists an $\overline{a}$-harmonic function $v^\ve\in H^1(B_{\nicefrac{1}{2}})$ such that, for every $\rho\in(0,\nicefrac{1}{4})$,  for $q_d=d\vee(2+\d)$ and $\nicefrac{1}{2_*}=\nicefrac{1}{2}-\nicefrac{1}{q_d}$,
\begin{align*}
& \fint_{B_{\nicefrac{R}{4}}}\abs{\nabla\left( u - v^\ve-\phi_i\partial_iv^\ve\right)}^2
\\ \nonumber & \leq c \left(\ve+\ve^{1-\frac{d-1}{q_d}}\left(\fint_{B_R}\abs{s}^{q_d}\right)^\frac{1}{q_d}\right)\fint_{B_R}\abs{\nabla u}^2
\\ \nonumber & \quad + c\ve^{-(d-1)}\rho^\frac{1}{2_*}\left(1+\left(\fint_{B_R}\abs{s}^{q_d}\right)^\frac{2}{q_d}\right)\fint_{B_R}\abs{\nabla u}^2
\\ \nonumber & \quad + c\rho^{-2(d+1)}R^{-2}\left[\left(1+\left(\fint_{B_R}\abs{s}^{q_d}\right)^\frac{2}{q_d}\right)\left(\fint_{B_R}\abs{\phi_i}^{2_*}\right)^\frac{2}{2_*}+\left(\fint_{B_R}\abs{\sigma_i}^2\right)\right]\fint_{B_R}\abs{\nabla u}^2.
\end{align*}
\end{prop}
\begin{proof}  We will first consider the case $R=1$ and obtain the general result by scaling.  Let $u\in H^1_{\textrm{loc}}(B_1)$ be an arbitrary distributional solution of the equation
\[-\nabla\cdot (a+s)\cdot \nabla u=0\;\;\textrm{in}\;\; B_1.\]
We will first prove that for every $\ve\in(0,1)$ there exists an $\overline{a}$-harmonic function $v^\ve\in H^1(B_{\nicefrac{1}{2}})$ such that the homogenization error $w^\ve=u-v^\ve-\phi_i\partial_iv^\ve$ satisfies, for every $\rho\in(0,\nicefrac{1}{4})$,
\begin{align}\label{lsr_0}
& \int_{B_{\nicefrac{1}{4}}}\abs{\nabla w^\ve}^2  \lesssim \left(\ve+\ve^{1-\frac{d-1}{q_d}}\left(\int_{B_1}\abs{s}^{q_d}\right)^\frac{1}{q_d}\right)\int_{B_1}\abs{\nabla u}^2
\\ \nonumber & \quad + \ve^{-(d-1)}\rho^\frac{1}{2_*}\left(1+\left(\int_{B_r}\abs{s}^{q_d}\right)^\frac{2}{q_d}\right)\int_{B_1}\abs{\nabla u}^2
\\ \nonumber & \quad + \rho^{-2(d+1)}\left[\left(1+\left(\int_{B_r}\abs{s}^{q_d}\right)^\frac{2}{q_d}\right)\left(\int_{B_r}\abs{\phi_i}^{2_*}\right)^\frac{2}{2_*}+\left(\int_{B_r}\abs{\sigma_i}^2\right)\right]\int_{B_1}\abs{\nabla u}^2,
\end{align}
for the correctors $\phi_i$ defined in \eqref{hom_03}.  Using Fubini's theorem, fix $r\in(\nicefrac{1}{2},\nicefrac{3}{4})$ such that
\begin{equation}\label{lsr_1}\int_{\partial B_r}\abs{\nabla u}^2\leq 4\int_{B_1}\abs{\nabla u}^2\;\;\textrm{and}\;\;\int_{\partial B_r}\abs{s}^{q_d}\leq 4\int_{B_1}\abs{s}^{q_d},\end{equation}
and for every $\ve\in(0,1)$ let $u^\ve$ denote a standard convolution of scale $\ve$ of $u$ on $\partial B_r$.  For each $\ve\in(0,1)$ let $v^\ve \in H^1(B_r)$ solve
\begin{equation}\label{lsr_2}-\nabla\cdot \overline{a}\nabla v^\ve=0\;\;\textrm{in}\;\;B_r\;\;\textrm{with}\;\;v^\ve=u^\ve\;\;\textrm{on}\;\;\partial B_r.\end{equation}
It then follows from Dirichlet-to-Neumann estimates Fabes, Jodeit and Rivi\`ere \cite[Theorem~2.4]{FabJodRiv1978} and Stein \cite[Chapter~7]{Ste1993}, \eqref{lsr_1}, and the fact that the convolution preserves the $L^2$-norm that
\[\int_{\partial B_r}\abs{\nu\cdot \nabla v^\ve} \lesssim \int_{\partial B_r}\abs{\nabla^{\textrm{tan}}v^\ve}^2 = \int_{\partial B_r}\abs{\nabla^{\textrm{tan}}u^\ve}^2\leq \int_{\partial B_r}\abs{\nabla u}^2\lesssim \int_{B_1}\abs{\nabla u}^2,\]
for the outward unit normal $\nu$ to $\partial B_r$.  Finally, for each $\rho\in(0,\nicefrac{1}{4})$ let $\eta_\rho\colon\R^d\rightarrow[0,1]$ be a smooth function satisfying $\eta_\rho=1$ on $\overline{B}_{1-\rho}$, satisfying $\eta_\rho = 0$ on $\R^d\setminus B_{1-\nicefrac{\rho}{2}}$, and satisfying $\abs{\eta_\rho(x)}\leq \nicefrac{c}{\rho}$ for some $c\in(0,\infty)$ independent of $\rho\in(0,\nicefrac{1}{4})$.  The first step will be to estimate the energy of the homogenization error $w^{\ve,\rho}\in H^1_0(B_r)$ defined by
\[w^{\ve,\rho}=u-v^\ve-\phi_i\eta_\rho\partial_i v^\ve.\]
A repetition of the derivation leading to \eqref{sh_32} proves that
\[-\nabla\cdot(a+s)\nabla w^{\ve,\rho} = \nabla\cdot\left[(1-\eta_\rho)((a+s)-\overline{a})\nabla v^\ve\right]+\nabla\cdot\left[\left(\phi_i(a+s)-\sigma_i\right)\nabla(\eta_\rho\partial_iv^\ve)\right]\]
in $B_r$ with boundary condition $w^{\ve,\rho}=u-u^\ve$ on $\partial B_r$, for the flux correctors $\sigma_i$ defined in Proposition~\ref{prop_flux_corrector}.  It follows from H\"older's inequality, Young's inequality, the triangle inequality, the uniform ellpticity, the definition of $\eta_\rho$, and a repetition of the argument leading to \eqref{hom_9} that
\begin{align} \label{lsr_5} & \int_{B_r}\abs{\nabla w^{\ve,\rho}}^2 \lesssim \abs{\int_{\partial B_r} (u-u^\ve)\nu\cdot \left((a+s)\nabla u-\overline{a}\nabla v\right)}
\\ \nonumber & +\rho^{\frac{1}{2_*}}\left(1+\left(\int_{B_r}\abs{s}^{q_d}\right)^\frac{2}{q_d}\right)\left(\int_{B_r}\abs{\nabla v^\ve}^{2\cdot 2_*}\right)^\frac{1}{2_*}
\\ \nonumber &+ \sup_{B_{(1-\nicefrac{\rho}{2})}}(\abs{\nabla(\partial_i v^\ve)}^2+\frac{1}{\rho^2}\abs{\partial_iv^\ve}^2)[(1+(\int_{B_r}\abs{s}^{q_d})^\frac{2}{q_d})(\int_{B_r}\abs{\phi_i}^{2_*})^\frac{2}{2_*}+(\int_{B_r}\abs{\sigma_i}^2)].
\end{align}
H\"older's inequality proves that, for the first term on the righthand side of \eqref{lsr_5}, 
\begin{align}\label{lsr_006}
& \abs{\int_{\partial B_r} (u-u^\ve)\nu\cdot \left((a+s)\nabla u-\overline{a}\nabla v\right)}
\\ \nonumber & \lesssim \left(\left(\int_{\partial B_r}\abs{\nabla u}^2\right)^\frac{1}{2}+\left(\int_{\partial B_r}\abs{\nabla v}^2\right)^\frac{1}{2}\right)\left(\int_{\partial B_r}\abs{u-u^\ve}^{2}\right)^\frac{1}{2}
\\ \nonumber & \quad + \left(\int_{\partial B_r}\abs{s}^{q_d}\right)^\frac{1}{q_d}\left(\int_{\partial B_r}\abs{\nabla u}^2\right)^\frac{1}{2}\left(\int_{\partial B_r}\abs{u-u^\ve}^{2_*}\right)^\frac{1}{2_*}.
\end{align}
Since for each $p\in[1,\infty)$ we have the convolution estimate
\[\left(\int_{\partial B_r}\abs{u-u^\ve}^p\right)^\frac{1}{p}\lesssim \ve\left(\int_{\partial B_r}\abs{\nabla^{\textrm{tan}} u}^p\right)^\frac{1}{p},\]
it follows from \eqref{lsr_1} that the first term on the righthand side of \eqref{lsr_006} is bounded by
\begin{equation}\label{lsr_00006}\left(\left(\int_{\partial B_r}\abs{\nabla u}^2\right)^\frac{1}{2}+\left(\int_{\partial B_r}\abs{\nabla v}^2\right)^\frac{1}{2}\right)\left(\int_{\partial B_r}\abs{u-u^\ve}^{2}\right)^\frac{1}{2}\lesssim \ve\int_{\partial B_r}\abs{\nabla u}^2 \lesssim \ve\int_{B_1}\abs{\nabla u}^2.\end{equation}
For the second term on the righthand side of \eqref{lsr_006}, since the definition of the convolution proves that $\fint_{B_r}(u-u^\ve)=0$, it follows from the Sobolev inequality, the fact that the convolution does not increase $L^p$-norms for $p\in[1,\infty)$, and the triangle inequality that
\begin{equation}\label{lsr_6}\left(\int_{\partial B_r}\abs{u-u^\ve}^{2_*}\right)^\frac{1}{2_*}\lesssim \left(\int_{\partial B_r}\abs{\nabla^{\textrm{tan}} (u-u^\ve)}^q\right)^\frac{1}{q}\lesssim \left(\int_{\partial B_r}\abs{\nabla u}^q\right)^\frac{1}{q},\end{equation}
for $q\in(1,2)$ defined by $\frac{1}{q}=\frac{1}{2_*}+\frac{1}{d-1}$.  Interpolating between the the convolution estimate with $p=2_*$ and \eqref{lsr_6} proves with \eqref{lsr_1} that
\begin{equation}\label{lsr_8} \left(\int_{\partial B_r}\abs{u-u^\ve}^{2_*}\right)^\frac{1}{2_*}\lesssim \ve^{1-\frac{d-1}{q_d}}\left(\int_{\partial B_r}\abs{\nabla u}^2\right)^\frac{1}{2}\lesssim \ve^{1-\frac{d-1}{q_d}}\left(\int_{B_r}\abs{\nabla u}^2\right)^\frac{1}{2}.\end{equation}
In combination \eqref{lsr_1} and \eqref{lsr_8} prove that the second term on the righthand side of \eqref{lsr_006} satisfies
\[\left(\int_{\partial B_r}\abs{s}^{q_d}\right)^\frac{1}{q_d}\left(\int_{\partial B_r}\abs{\nabla u}^2\right)^\frac{1}{2}\left(\int_{\partial B_r}\abs{u-u^\ve}^{2_*}\right)^\frac{1}{2_*}\lesssim \ve^{1-\frac{d-1}{q_d}}\left(\int_{B_1}\abs{s}^{q_d}\right)^\frac{1}{q_d}\int_{B_1}\abs{\nabla u}^2.\]
Returning to \eqref{lsr_006}, it follows from \eqref{lsr_00006} that
\begin{equation}\label{lsr_10}  \abs{\int_{\partial B_r} (u-u^\ve)\nu\cdot \left((a+s)\nabla u-\overline{a}\nabla v\right)}\lesssim \left(\ve+\ve^{1-\frac{d-1}{q_d}}\left(\int_{B_1}\abs{s}^{q_d}\right)^\frac{1}{q_d}\right)\int_{B_1}\abs{\nabla u}^2.\end{equation}
For the second term on the righthand side of \eqref{lsr_5}, since it follows from H\"older's inequality, the definition of the convolution kernel, and \eqref{lsr_1} that
\begin{equation}\label{lsr_11}\sup_{\partial B_r}\abs{\nabla^{\textrm{tan}} u^\ve}\lesssim \ve^{-\frac{d-1}{2}}\left(\int_{\partial B_r}\abs{\nabla^{\textrm{tan}} u}^2\right)^\frac{1}{2}\lesssim \ve^{-\frac{d-1}{2}}\left(\int_{B_1}\abs{\nabla u}^2\right)^\frac{1}{2},\end{equation}
it follows from Proposition~\ref{prop_ue_int} and \eqref{lsr_11} that
\begin{equation}\label{lsr_12}\left(\int_{B_r}\abs{\nabla v^\ve}^{2\cdot 2_*}\right)^\frac{1}{2_*}\lesssim \ve^{-(d-1)}\int_{B_1}\abs{\nabla u}^2.\end{equation}
Therefore, using \eqref{lsr_12}, the second term on the righthand side of \eqref{lsr_5} is bounded by
\begin{align}\label{lsr_14}& \rho^{\frac{1}{2_*}}\left(1+\left(\int_{B_r}\abs{s}^{q_d}\right)^\frac{2}{q_d}\right)\left(\int_{B_r}\abs{\nabla v^\ve}^{2\cdot 2_*}\right)^\frac{1}{2_*}
\\ \nonumber & \lesssim \ve^{-(d-1)}\rho^\frac{1}{2_*}\left(1+\left(\int_{B_r}\abs{s}^{q_d}\right)^\frac{2}{q_d}\right)\int_{B_1}\abs{\nabla u}^2.\end{align}
For the final term of \eqref{lsr_5}, it follows from Proposition~\ref{prop_ue_int} that
\begin{align}\label{lsr_15}
& \sup_{B_{(1-\nicefrac{\rho}{2})}}(\abs{\nabla(\partial_i v^\ve)}^2+\frac{1}{\rho^2}\abs{\partial_iv^\ve}^2)[(1+(\int_{B_r}\abs{s}^{q_d})^\frac{2}{q_d})(\int_{B_r}\abs{\phi_i}^{2_*})^\frac{2}{2_*}+(\int_{B_r}\abs{\sigma_i}^2)]
\\ \nonumber & \lesssim \rho^{-2(d+1)}\left[\left(1+\left(\int_{B_r}\abs{s}^{q_d}\right)^\frac{2}{q_d}\right)\left(\int_{B_r}\abs{\phi_i}^{2_*}\right)^\frac{2}{2_*}+\left(\int_{B_r}\abs{\sigma_i}^2\right)\right]\int_{B_1}\abs{\nabla u}^2.
\end{align}
For every $\ve\in(0,1)$ let $w^\ve=u-v^\ve-\phi_i\partial_i v^\ve$.  It follows from $\rho\in(0,\nicefrac{1}{4})$ and the definition of $\eta_\rho$ that $w^\ve=w^{\ve,\rho}$ in $B_{\nicefrac{1}{4}}$ and therefore it follows from \eqref{lsr_5}, \eqref{lsr_10}, \eqref{lsr_14}, and \eqref{lsr_15} that, for each $\rho\in(0,\nicefrac{1}{4})$,
\begin{align}\label{lsr_17}
& \int_{B_{\nicefrac{r}{4}}}\abs{\nabla w^\ve}^2  \lesssim \left(\ve+\ve^{1-\frac{d-1}{q_d}}\left(\int_{B_1}\abs{s}^{q_d}\right)^\frac{1}{q_d}\right)\int_{B_1}\abs{\nabla u}^2
\\ \nonumber & \quad + \ve^{-(d-1)}\rho^\frac{1}{2_*}\left(1+\left(\int_{B_1}\abs{s}^{q_d}\right)^\frac{2}{q_d}\right)\int_{B_1}\abs{\nabla u}^2
\\ \nonumber & \quad + \rho^{-2(d+1)}\left[\left(1+\left(\int_{B_1}\abs{s}^{q_d}\right)^\frac{2}{q_d}\right)\left(\int_{B_1}\abs{\phi_i}^{2_*}\right)^\frac{2}{2_*}+\left(\int_{B_1}\abs{\sigma_i}^2\right)\right]\int_{B_1}\abs{\nabla u}^2,
\end{align}
which completes the proof of \eqref{lsr_0} with $v^\ve\in H^1(B_{\nicefrac{1}{2}})$ defined by \eqref{lsr_2}.  It then follows by scaling that, for each $R\in(0,\infty)$, for any $u\in H^1(B_R)$ that is a weak solution of
\[-\nabla\cdot (a+s)\cdot \nabla u=0\;\;\textrm{in}\;\; B_R,\]
there exists for every $\ve\in(0,1)$ an $\overline{a}$-harmonic function $v^\ve\in H^1(B_{\nicefrac{R}{2}})$ such that the homogenization error $w=u-v^\ve-\phi_i\partial_i v^\ve$ satisfies, for every $\rho\in(0,\nicefrac{1}{4})$, for some $\overline{c}\in(0,\infty)$ independent of $R$, $\ve$, and $\rho$,
\begin{align}\label{lsr_18}
& \fint_{B_{\nicefrac{R}{4}}}\abs{\nabla w^\ve}^2  \leq\overline{c} \left(\ve+\ve^{1-\frac{d-1}{q_d}}\left(\fint_{B_R}\abs{s}^{q_d}\right)^\frac{1}{q_d}\right)\fint_{B_R}\abs{\nabla u}^2
\\ \nonumber & \quad + \overline{c}\ve^{-(d-1)}\rho^\frac{1}{2_*}\left(1+\left(\fint_{B_R}\abs{s}^{q_d}\right)^\frac{2}{q_d}\right)\fint_{B_R}\abs{\nabla u}^2
\\ \nonumber & \quad + \overline{c}\rho^{-2(d+1)}R^{-2}\left[\left(1+\left(\fint_{B_R}\abs{s}^{q_d}\right)^\frac{2}{q_d}\right)\left(\fint_{B_R}\abs{\phi_i}^{2_*}\right)^\frac{2}{2_*}+\left(\fint_{B_R}\abs{\sigma_i}^2\right)\right]\fint_{B_R}\abs{\nabla u}^2.
\end{align}
The proof follows by considering the rescalings $u^R(x)=R^{-1}u(Rx)$, $\phi^R_i(x)=R^{-1}\phi_i(Rx)$, and $\sigma_i^R(x)=R^{-1}\sigma_i(x)$ and by repeating the argument leading to \eqref{lsr_17} to obtain an $\overline{a}$-harmonic function $\tilde{v}^\ve$ in $H^1(B_{\nicefrac{1}{2}})$ such that the homogenization error $w^{R,\ve}(x)=u^R-v^\ve-\phi_i^R\partial_i\tilde{v}^\ve$ satisfies \eqref{lsr_17} with $\phi_i^R$, $\sigma_i^R$, and $s^R(x)=s(Rx)$.  We then define $v^\ve(x)=R\tilde{v}^\ve(\nicefrac{x}{R})$ and obtain \eqref{lsr_18} from \eqref{lsr_17} after rescaling.  This completes the proof. \end{proof}

\begin{prop}\label{prop_excess} Assume \eqref{steady} and let $q_d=d\vee(2+\d)$.  For every $\a\in(0,1)$ there exist $c_\a,C_\a\in(0,\infty)$ such that if for any $R_1<R_2\in(0,\infty)$ we have, for every $R\in[R_1,R_2]$ and $i\in\{1,\ldots,d\}$,
\[\frac{1}{R}\left(\fint_{B_R}\abs{\phi_i}^{2_*}\right)^\frac{1}{2_*}+\frac{1}{R}\left(\fint_{B_R}\abs{\sigma_i}^2\right)^\frac{1}{2}\leq \nicefrac{1}{C_\a},\]
and we have, for every $R\in[R_1,R_2]$,
\[\left(\fint_{B_R}\abs{s}^{q_d}\right)^\frac{1}{q_d}\leq\E\left[\abs{S}^{q_d}\right]^\frac{1}{q_d}+1\;\;\textrm{and}\;\;\sum_{i=1}^d\fint_{B_R}\abs{\Phi_i+e_i}^2\leq \sum_{i=1}^d\E\left[\abs{\Phi_i+e_i}^2\right]+1,\]
then every weak solution $u\in H^1_{\textrm{loc}}(\R^d)$ of the equation
\[-\nabla\cdot (a+s)\nabla u = 0 \;\;\textrm{in}\;\;\R^d,\]
satisfies
\[R_1^{-2\a}\Exc(u;R_1)\leq c_\a R_2^{-2\a}\Exc(u;R_2).\]
\end{prop}

\begin{proof}  Let $u\in H^1_{\textrm{loc}}(\R^d)$ be a distributional solution of $-\nabla\cdot (a+s)\nabla u = 0$ in $\R^d$, let $R\in(0,\infty)$, and for each $\ve\in(0,1)$ let $w^\ve$ be defined in \eqref{lsr_18} on $B_{\nicefrac{R}{2}}$.  We will first show that, for each $\delta_0\in(0,1)$, there exists $C_1\in(1,\infty)$ depending on $\delta_0$ such that if
\[\left(\fint_{B_R}\abs{s}^{q_d}\right)^\frac{1}{q_d}\leq \E\left[\abs{S}^{q_d}\right]^\frac{1}{q_d}+1,\]
and such that if
\[\sum_{i=1}^d\left(\frac{1}{R^2}\left(\fint_{B_R}\abs{\phi_i}^{2_*}\right)^\frac{2}{2_*}+\frac{1}{R^2}\fint_{B_R}\abs{\sigma_i}^2\right)\leq \frac{1}{C_1},\]
then there exists a deterministic $\ve_0\in(0,1)$ depending on $\delta_0$ such that
\begin{equation}\label{lsr_19}\fint_{B_{\nicefrac{R}{4}}}\abs{\nabla w^{\ve_0}}^2\leq \delta_0\fint_{B_R}\abs{\nabla u}^2,\end{equation}
for $w^{\ve_0}$ defined in \eqref{lsr_18}.  First fix $\ve_0\in(0,1)$ such that
\begin{equation}\label{lsr_20}\overline{c}\left(\ve_0+\ve_0^{1-\frac{d-1}{q_d}}\left(\E\left[\abs{S}^{q_d}\right]^{\frac{2}{q_d}}+1\right)\right)<\nicefrac{\delta_0}{3}.\end{equation}
Then fix $\rho_0\in(0,\nicefrac{1}{4})$ such that
\begin{equation}\label{lsr_21}\overline{c}\left(\ve_0^{-(d-1)}\rho_0^\frac{1}{2_*}\left(2+\E\left[\abs{S}^{q_d}\right]^\frac{2}{q_d}\right)\right)<\nicefrac{\delta_0}{3}.\end{equation}
Finally fix $C_1\in(1,\infty)$ satisfying
\begin{equation}\label{lsr_22}\overline{c}\rho_0^{-2(d+1)}\left(2+\E\left[\abs{S}^{q_d}\right]^\frac{2}{q_d}\right)< \nicefrac{C_1\delta_0}{3}.\end{equation}
The claim \eqref{lsr_19} then follows from Proposition~\ref{prop_hom_energy}, \eqref{lsr_20}, \eqref{lsr_21}, and \eqref{lsr_22}.

We will now prove that there exists $\theta_0\in(0,\nicefrac{1}{4})$ and $C_2\in(1,\infty)$ such that for any $R\in(0,\infty)$ satisfying, for every $r\in[\theta_0R,R]$,
\begin{equation} \label{lsr_103}
\sum_{i=1}^d\frac{1}{r}\left(\fint_{B_r}\abs{\phi_i}^{2_*}\right)^\frac{1}{2_*}+\frac{1}{r}\left(\fint_{B_{r}}\abs{\sigma_i}^2\right)^\frac{1}{2}\leq \nicefrac{1}{C_2}, \end{equation}
and, for every $r\in[\theta_0R,R]$,
\begin{equation}\label{lsr_104}\left(\fint_{B_{r}}\abs{s}^{q_d}\right)^\frac{1}{q_d}\leq\E\left[\abs{S}^{q_d}\right]^\frac{1}{q_d}+1\;\;\textrm{and}\;\;\sum_{i=1}^d\fint_{B_{r}}\abs{\Phi_i+e_i}^2\leq \sum_{i=1}^d\E\left[\abs{\Phi_i+e_i}^2\right]+1,\end{equation}
we have the exact inequality
\[(\theta_0R)^{-2\a}\Exc(u;\theta_0R)\leq R^{-2\a}\Exc(u;R).\]
For each $\ve\in(0,1)$ let $\xi^\ve=\nabla v^{\ve}(0)\in\R^d$ for $v^{\ve}$ defined in \eqref{lsr_18} and observe that
\[u-\xi^\ve-\nabla\phi_{\xi^\ve} = \nabla w^{\ve}+(\nabla v^{\ve}-\nabla v^{\ve}(0))+\nabla \phi_i(\partial_iv^{\ve}-\partial_iv^{\ve}(0))+\phi_i\nabla(\partial_iv^{\ve})\;\;\textrm{on}\;\;B_{\nicefrac{R}{2}},\]
for $w^\ve$ defined in \eqref{lsr_18}.  For every $r\in(0,\nicefrac{R}{4})$ the triangle inequality, Young's inequality, and the mean value theorem prove that
\begin{align*}
& \int_{B_r}\abs{u-\xi^\ve-\nabla\phi_{\xi^\ve}}^2
\\ & \lesssim \int_{B_{\nicefrac{R}{4}}}\abs{\nabla w^{\ve}}^2+r^2\sup_{B_r}\abs{\nabla\partial_i v^{\ve}}^2\int_{B_r}\abs{e_i+\nabla \phi_i}^2+\sup_{B_r}\abs{\nabla \partial_iv^{\ve}}^2\int_{B_r}\abs{\phi_i}^2.
\end{align*}
Proposition~\ref{prop_ue_int} proves that
\[\int_{B_r}\abs{u-\xi^\ve-\nabla\phi_{\xi^\ve}}^2\lesssim  \int_{B_{\nicefrac{R}{4}}}\abs{\nabla w^{\ve}}^2+\left(\frac{r^2}{R^2}\int_{B_r}\abs{e_i+\nabla \phi_i}^2+\frac{1}{R^2}\int_{B_r}\abs{\phi_i}^2\right)\fint_{B_R}\abs{\nabla u}^2.\]
After dividing by $r^d$, for some $\overline{c}\in(0,\infty)$ independent of $\ve$, $r$, and $R$,
\begin{align*}
& \fint_{B_r}\abs{u-\xi^\ve-\nabla\phi_{\xi^\ve}}^2
\\ & \leq\overline{c}\left(\frac{R^d}{r^d}\fint_{B_{\nicefrac{R}{4}}}\abs{\nabla w^{\ve}}^2+ \left(\frac{r^2}{R^2}\left(\fint_{B_r}\abs{e_i+\nabla \phi_i}^2+\frac{1}{R^2}\fint_{B_r}\abs{\phi_i}^2\right)\right)\fint_{B_R}\abs{\nabla u}^2\right).
\end{align*}
Since $\a\in(0,1)$ fix $\theta_0\in(0,\nicefrac{1}{4})$ such that
\begin{equation}\label{lsr_37}\overline{c}\theta_0^2\left(\sum_{i=1}^d\E\left[\abs{\Phi_i+e_i}^2\right]+2\right)\leq \nicefrac{\theta_0^{2\a}}{2},\end{equation}
then fix $\delta_0\in(0,1)$ such that
\begin{equation}\label{lsr_38}\overline{c}\theta_0^{-d}\delta_0\leq \nicefrac{\theta_0^{2\a}}{2},\end{equation}
and let $C_2\in(1,\infty)$ and $\ve_0\in(0,1)$ satisfy the conclusion of \eqref{lsr_19} for this $\delta_0$.  The definition of the excess, $C_2\in(1,\infty)$, \eqref{lsr_19}, \eqref{lsr_103}, and \eqref{lsr_104} then prove after choosing $r=\theta_0R$ that
\begin{align}\label{lsr_36}
\Exc(u;\theta_0 R) & \leq \fint_{B_{\theta_0 R}}\abs{u-\xi^{\ve_0}-\nabla\phi_{\xi^{\ve_0}}}^2
\\ & \leq\overline{c} \left(\delta_0\theta_0^{-d}+\theta_0^2\left(\sum_{i=1}^d\E\left[\abs{\Phi_i+e_i}^2\right]+2\right)\right)\fint_{B_R}\abs{\nabla u}^2.\nonumber
\end{align}
In combination \eqref{lsr_37}, \eqref{lsr_38}, and \eqref{lsr_36} prove that
\begin{equation}\label{lsr_40}\Exc(u;\theta_0R)\leq \theta_0^{2\alpha}\fint_{B_R}\abs{\nabla u}^2.\end{equation}
We now observe that for every $\xi\in\R^d$ the function $u-\xi\cdot x-\phi_\xi\in H^1_{\textrm{loc}}(\R^d)$ solves
\begin{equation}\label{lsr_41}-\nabla\cdot(a+s)\nabla (u-\xi\cdot x-\phi_\xi)=0\;\;\textrm{in}\;\;\R^d,\end{equation}
and by definition of the excess and linearity we have
\[\Exc(u;\theta_0 R)=\Exc(u-\xi\cdot x-\phi_\xi;\theta_0 R)\;\;\textrm{for every}\;\;\xi\in\R^d.\]
Therefore, since $u\in H^1_{\textrm{loc}}(\R^d)$ solving \eqref{lsr_41} was arbitrary, we have from \eqref{lsr_40} and the definition of the excess that
\begin{align}\label{lsr_42}\Exc(u;\theta_0R) & =\inf_{\xi\in\R^d}\Exc(u-\xi\cdot x-\phi_\xi;\theta_0R)
\\ \nonumber & \leq \theta_0^{2\alpha}\inf_{\xi\in\R^d}\left(\fint_{B_R}\abs{\nabla u-\xi-\nabla\phi_\xi}^2\right)=\theta_0^{2\a}\Exc(u;R).\end{align}
We will now use the exact inequality \eqref{lsr_42} to conclude.  For $R_1\leq R_2\in(0,\infty)$ suppose that \eqref{lsr_103} and \eqref{lsr_104} are satisfied for every $r\in[R_1,R_2]$.  We will prove that there exists $c\in(0,\infty)$ depending on $\a$ but independent of $R_1,R_2\in(0,\infty)$ such that
\[\Exc(u;R_1)\leq c\left(\nicefrac{R_1}{R_2}\right)^{2\a}\Exc(u;R_2).\]
For $\theta_0\in(0,\nicefrac{1}{4})$ defined in \eqref{lsr_38}, if $\nicefrac{R_1}{R_2}\geq\theta_0$ then by definition of the excess
\[\Exc(u;R_1)\leq \left(\nicefrac{R_2}{R_1}\right)^d\Exc(u;R_2)\leq\theta_0^{-d}\Exc(u;R_2)\leq \theta_0^{-(d+2\a)}\left(\nicefrac{R_1}{R_2}\right)^{2\a}\Exc(u;R_2).\]
If $\nicefrac{R_1}{R_2}<\theta_0$ let $n\in\N$ be the unique positive integer satisfying $\theta_0^{n}\leq\nicefrac{R_1}{R_2}<\theta_0^{n-1}$ and observe by induction, the previous step, and \eqref{lsr_42} that
\begin{align*}
\Exc(u;R_1)& \leq \theta_0^{-(d+2\a)}\Exc(u;\theta_0^{n-1}R_2) \leq \theta_0^{-(d+2\a)}\theta_0^{2\a}\Exc(u;\theta_0^{n-2}R_2)
\\ & \leq \theta_0^{-(d+2\a)}(\theta_0^{n-1})^{2\a}\Exc(u;R_2)\leq \theta_0^{-(d+4\a)}(\nicefrac{R_1}{R_2})^{2\a}\Exc(u;R_2),
\end{align*}
which completes the proof.
\end{proof}

\begin{thm}\label{thm_excess}  Assume~\eqref{steady}.  On a subset of full probability, for every $\a\in(0,1)$ there exists a deterministic $c\in(0,\infty)$ and a random radius $R_0\in(0,\infty)$ such that, for every weak solution $u\in H^1_{\textrm{loc}}(\R^d)$ of the equation
\[-\nabla\cdot (a+s)\nabla u = 0 \;\;\textrm{in}\;\;\R^d,\]
for every $R_1<R_2\in(R_0,\infty)$,
\[R_1^{-2\a}\Exc(u;R_1)\leq cR_2^{-2\a}\Exc(u;R_2).\]
\end{thm}

\begin{proof}  The proof is an immediate consequence of the ergodic theorem, Proposition~\ref{prop_sublinear}, and Proposition~\ref{prop_excess}.\end{proof}

\section{The Liouville theorem}\label{sec_lvl}  In this section, we will prove the first-order Liouville theorem for subquadratic solutions $u\in H^1_{\textrm{loc}}(\R^d)$ of the equation
\[-\nabla\cdot(a+s)\nabla u = 0\;\;\textrm{in}\;\;\R^d.\]
That is, in analogy with the Liouville theorem for harmonic functions on Euclidean space, the space of subquadratic $(a+s)$-harmonic functions is spanned by the $(a+s)$-harmonic coordinates.  The section is organized as follows.  We prove in Proposition~\ref{prop_lvl_est} below a version of the Caccioppoli inequality adapted to the divergence-free setting.  We then prove the Liouville theorem in Theorem~\ref{thm_lvl} below, which is a consequence of the large-scale regularity estimate of Theorem~\ref{thm_excess} and the Caccioppoli inequality.  These methods are motivated by the analogous results in \cite{BelFehOtt2018,GloNeuOtt2020} from the elliptic setting.

\begin{prop}\label{prop_lvl_est}  Assume \eqref{steady}.  Let $q_d=d\vee(2+\d)$, let $\nicefrac{1}{2_*}=\nicefrac{1}{2}-\nicefrac{1}{q_d}$, and let  $u\in H^1_{\textrm{loc}}(\R^d)$ be a weak solution of
\begin{equation}\label{lvl_1}-\nabla\cdot(a+s)\nabla u=0\;\;\textrm{in}\;\;\R^d.\end{equation}
Then, for every $R\in(0,\infty)$, for some $c\in(0,\infty)$ independent of $R$,
\[\fint_{B_R}\abs{\nabla u}^2\leq \frac{c}{R^2}\left(\fint_{B_{2R}}\abs{u}^2+\left(\fint_{B_{2R}}\abs{s}^{q_d}\right)^{\frac{2}{q_d}}\left(\fint_{B_{2R}}\abs{u}^{2_*}\right)^\frac{2}{2_*}\right).\]
\end{prop}

\begin{proof}  Let $\eta\colon\R^d\rightarrow\R$ be a smooth cutoff function satisfying $\eta=1$ on $\overline{B}_1$, satisfying $\eta=0$ on $\R^d\setminus B_2$, and for each $R\in(0,\infty)$ define $\eta_R(x)=\eta(\nicefrac{x}{R})$.  A repetition of the argument leading to \eqref{hom_9} proves that, after testing \eqref{lvl_1} with $\eta_R^2 u_n$ for $u_n=(u\wedge n)\vee(-n)$ and passing to the limit $n\rightarrow\infty$,
\[\int_{B_R}a\nabla u\cdot\nabla u \eta^2_R = -2\int_{B_R}a\nabla u\cdot\nabla\eta_R u\eta_R -2\int_{B_R}s\nabla u\cdot \nabla\eta_R u\eta_R.\]
It then follows by definition of $\eta_R$, the uniform ellipticity, H\"older's inequality, and Young's inequality that, for some $c\in(0,\infty)$ independent of $R$,
\[\fint_{B_R}\abs{\nabla u}^2\leq \frac{c}{R^2}\left(\fint_{B_{2R}}\abs{u}^2+\left(\fint_{B_{2R}}\abs{s}^{q_d}\right)^{\frac{2}{q_d}}\left(\fint_{B_{2R}}\abs{u}^{2_*}\right)^\frac{2}{2_*}\right).\qedhere\]
\end{proof}

\begin{thm}\label{thm_lvl}  Assume \eqref{steady}.  Let $q_d=d\vee(2+\d)$ and $\nicefrac{1}{2_*}=\nicefrac{1}{2}-\nicefrac{1}{q_d}$.  Then almost surely every weak solution $u\in H^1_{\textrm{loc}}(\R^d)$ of the equation
\begin{equation}\label{lvl_20}-\nabla\cdot(a+s)\nabla u=0\;\;\textrm{in}\;\;\R^d,\end{equation}
that is strictly subquadratic in the sense that, for some $\a\in(0,1)$,
\begin{equation}\label{lvl_021}\lim_{R\rightarrow\infty}\frac{1}{R^{1+\a}}\left(\fint_{B_R}\abs{u}^{2_*}\right)^\frac{1}{2_*}=0,\end{equation}
satisfies $u=c+\xi\cdot x+\phi_\xi$ on $\R^d$ for some $c\in(0,\infty)$ and $\xi\in\R^d$.
\end{thm}

\begin{proof}  By the ergodic theorem and $\E[\Phi_i]=0$ for each $i\in\{1,\ldots,d\}$ let $\O_1\subseteq\O$ be the subset of full probability satisfying, for every $\xi\in\R^d$,
\begin{equation}\label{lvl_19}\lim_{R\rightarrow\infty}\fint_{B_R}\abs{\nabla\phi_\xi+\xi}^2=\E\left[\abs{\Phi_\xi}^2\right]+\abs{\xi}^2\geq \abs{\xi}^2\;\;\textrm{and}\;\;\lim_{R\rightarrow\infty}\fint_{B_R}\abs{s}^{q_d}=\E\left[\abs{S}^{q_d}\right].\end{equation}
Let $\O_2\subseteq\O$ be the subset of full probability satisfying the conclusion Theorem~\ref{thm_excess}, and let $\O_3=\O_1\cap\O_2$.  For $\o\in\O_3$ let $R_0\in(0,\infty)$ be such that every weak solution $u\in H^1_{\textrm{loc}}(\R^d)$ of \eqref{lvl_20} satisfies, for every $R_1<R_2\in(R_0,\infty)$, for a deterministic $c\in(0,\infty)$ depending on $\a$,
\begin{equation}\label{lvl_21}R_1^{-2\a}\Exc(u;R_1)\leq c R_2^{-2\a}\Exc(u;R_2),\end{equation}
and such that, for every $R\in(R_0,\infty)$ and $\xi\in\R^d$,
\begin{equation}\label{level_22}\fint_{B_R}\abs{\nabla\phi_\xi+\xi}^2\geq\nicefrac{\abs{\xi}^2}{2}.\end{equation}
The definition of the excess, $u\in H^1_{\textrm{loc}}(\R^d)$, and \eqref{level_22} prove that, for every $R\in(R_0,\infty)$,
\begin{equation}\label{lvl_23}\Exc(u;R)=\inf_{\xi\in\R^d}\left(R^{-2\a}\fint_{B_R}{\nabla u-\xi-\nabla\phi_\xi}^2\right)=\min_{\xi\in\R^d}\left(R^{-2\a}\fint_{B_R}{\nabla u-\xi-\nabla\phi_\xi}^2\right).\end{equation}
Fix $R_1\in(R_0,\infty)$.  We have by definition of the excess, Proposition~\ref{prop_lvl_est}, and \eqref{lvl_21} that, for every $R\in(R_1,\infty)$,
\begin{align}
R_1^{-2\a}\Exc(u;R_1) & \leq c  R^{-2\a}\fint_{B_R}\abs{\nabla u}^2
\\ \nonumber & \leq c R^{-2(1+\a)}\left(\fint_{B_{2R}}\abs{u}^2+\left(\fint_{B_{2R}}\abs{s}^{q_d}\right)^{\frac{2}{q_d}}\left(\fint_{B_{2R}}\abs{u}^{2_*}\right)^\frac{2}{2_*}\right).
\end{align}
H\"older's inequality, the ergodic theorem, $2_*\in(2,\infty)$, \eqref{lvl_021}, and \eqref{lvl_19} prove almost surely that
\begin{align*}
& R_1^{-2\a}\Exc(u;R_1)
\\ & \leq c \limsup_{R\rightarrow\infty}R^{-2(1+\a)}\left(\fint_{B_{2R}}\abs{u}^2+\left(\fint_{B_{2R}}\abs{s}^{q_d}\right)^{\frac{2}{q_d}}\left(\fint_{B_{2R}}\abs{u}^{2_*}\right)^\frac{2}{2_*}\right)=0.
\end{align*}
It then follows from \eqref{lvl_23} that there exists $\xi_{R_1}\in\R^d$ such that
\[\nabla u-\xi_{R_1}-\nabla\phi_{\xi_{R_1}}=0\;\;\textrm{in}\;\;L^2(B_{R_1};\R^d),\]
and therefore there exists $c_{R_1}\in\R$ such that
\[u=c_{R_1}+\xi_{R_1}\cdot x+\phi_{\xi_{R_1}}\;\;\textrm{in}\;\;H^1(B_{R_1}).\]
Since the linearity and \eqref{level_22} prove that $c_{R_1}=c_{R_2}$ and $\xi_{R_1}=\xi_{R_2}$ whenever $R_1\leq R_2\in(R_0,\infty)$, there exists $\xi\in\R^d$ and $c\in\R$ such that
\[u=c+\xi\cdot x+\phi_\xi\;\;\textrm{in}\;\;H^1_{\textrm{loc}}(\R^d).\qedhere\]
\end{proof}


\bibliographystyle{plain} 
\bibliography{div_free}       

\end{document}